\documentclass[reqno,11pt]{amsart}

\usepackage[top=2.0cm,bottom=2.0cm,left=1.9cm,right=1.9cm]{geometry}
\usepackage{amsthm,amsmath,amssymb,dsfont}
\usepackage{mathrsfs,amsfonts,functan,extarrows,mathtools}
\usepackage[colorlinks]{hyperref}

\usepackage{marginnote}
\usepackage{xcolor}
\usepackage{stmaryrd}
\usepackage{esint}
\usepackage{graphicx}
\usepackage{bm}

\newtheorem{theorem}{Theorem}[section]
\newtheorem{proposition}{Proposition}[section]

\newtheorem{corollary}{Corollary}[section]
\newtheorem{remark}{Remark}[section]
\newtheorem{lemma}{Lemma}[section]

\numberwithin{equation}{section}
\allowdisplaybreaks

\def\R{\mathbb{R}}

\def\P{\mathbf{P}}
\def\I{\mathbf{I}}
\def\v{\varepsilon}
\def\div{\mathrm{div}}

\def\l{\langle}
\def\r{\rangle}

\def\up{\textup}
\def\n{\nabla}
\def\p{\partial}

\def\H{\mathcal{H}}

\def\a{\alpha}
\def\b{\beta}

\def\D{\mathcal{D}}
\def\c{\cdot}

\def\E{\mathcal{E}}
\def\M{\sqrt{M}}

\newcounter{wronumber}

\begin{document}

\title[The small Deborah number limit for the VFP-INS system]
{The small Deborah number limit for the fluid-particle flows: incompressible case}

\author[Zhendong Fang]{Zhendong Fang}
\address[Zhendong Fang]
        {\newline School of Mathematics, South China University of Technology, Guangzhou, 510641, P. R. China}
\email{zdfang@scut.edu.cn}

\author[Kunlun Qi]{Kunlun Qi}
\address[Kunlun Qi]
        {\newline School of Mathematics, University of Minnesota - Twin Cities, Minneapolis, MN, 55455, USA}
\email{kqi@umn.edu}

\author[Huanyao Wen]{Huanyao Wen}
\address[Huanyao Wen]
        {\newline School of Mathematics, South China University of Technology, Guangzhou, 510641, P. R. China}
\email{mahywen@scut.edu.cn}




\maketitle

\begin{abstract}

This work is devoted to the study of the hydrodynamic limit for the fluid-particle flows governed by the Vlasov-Fokker-Planck (VFP) equation coupled with the incompressible Navier-Stokes (INS) equation as the Deborah number approaches to zero. The limit is valid globally in time provided that the initial perturbation is small in a neighborhood of a steady state. The proof is based on a formal derivation of the limiting system via the Hilbert approach, followed by a rigorous justification via introducing a novel decomposition involving some macroscopic quantities and a refined energy estimate motivated by macro-micro decomposition. In contrast to the existing results for the same scaled model, the present work provides the first one on the hydrodynamic limits in a strong sense with an explicit convergence rate.

\end{abstract}


\phantomsection
\addcontentsline{toc}{section}{\contentsname}



\section{Introduction}
\label{sec:intro}

\subsection{The model}
\label{subsec:bg}

The fluid-particle coupled system has attracted a lot of attention in recent years, due to its wide application in describing the phenomena arising from physics, engineering, and medical research, for instance, the droplets of spray, diesel engines, sedimentation analysis, biotechnology, and mineral processes \cite{CBLBPEJSM05, SBRBKT03, WFA85}. Among those, the two-phase flow system consisting of the Vlasov-Fokker-Planck equation for particles coupled with the Navier-Stokes equation for fluid is of particular interest.
In this paper, we consider the hydrodynamic limit of the following scaled Vlasov-Fokker-Planck equation coupled with the incompressible Navier-Stokes equation (VFP-INS) system:
\begin{equation}\label{VFP-INS}
\left\{
\begin{aligned}
&\p_t f^\v+\frac{1}{\v}v\c\n_x f^\v=\frac{1}{\v^2}\div_v[\n_v f^\v+(v-\v u^\v)f^\v],\\[4pt]
&\p_t  u^\v+u^\v\c\n_x u^\v+\n_x p^\v=\Delta_x u^\v+\frac{1}{\v}\int_{\R^3} (v-\v u^\v) f^\v \,dv,\\[4pt]
&\up{div}_x u^\v=0,\\[4pt]
&(f^\v(t,x,v), u^\v(t,x))|_{t=0} = (f^{\v,in}(x,v),u^{\v,in}(x)),\\[4pt]
&(f^\v(t,x,v), u^\v(t,x)) \to (M, 0),\quad \up{as} \quad |x|\to +\infty,
\end{aligned}
\right.
\end{equation}
where $f^\v(t, x, v),\,u^\v(t,x),\,p^\v(t,x)$ denote the density distribution function of the particles, velocity of the fluid, and the pressure of the fluid at the time $t\geq 0$, position $x\in\R^3$, velocity $v\in\R^3$, respectively. $\v$ denotes the Deborah number, for which we refer the readers to \cite{RRH75} for more physical details. $M $ is a global normalized Maxwellian equilibrium given by
\begin{equation}\label{Maxwellian}
M:=M(v)=\frac{1}{(2\pi)^{\frac{3}{2}}}\up{e}^{-\frac{|v|^2}{2}}.
\end{equation}
This scaled system \eqref{VFP-INS} was first proposed in \cite{GTJPEVA04-1}, where the small Deborah number limit for light particles regime was defined in the sense that, as $\v \to 0$, it leads \eqref{VFP-INS} to the following macroscopic system, namely the Navier-Stokes equation coupled with the Kramer-Smoluchowski equation (NS-KS) system:
\begin{equation}\label{NS-KS}
\left\{
\begin{aligned}
&\p_t n_0+\up{div}_x(u_0 n_0)=\Delta_x n_0,\\[4pt]
&\p_t u_0+u_0\c\n_x u_0+\n_x p_0=\Delta_x u_0,\\[4pt]
&\up{div}_x u_0=0,\\[4pt]
&(n_0(t,x), u_0(t,x))|_{t=0}=(n_0^{in}(x),u_0^{in}(x)),\\[4pt]
&(n_0(t,x), u_0(t,x)) \to (1,0), \quad \up{as} \quad |x|\to +\infty.
\end{aligned}
\right.
\end{equation}

\subsection{Previous results and our contributions}
\label{subsec:previous}

In light of its paramount physical significance, the fluid-particle two-phase model has been extensively studied for a long time. We first review the previous literature concerning its well-posedness and hydrodynamic limit, and the main purpose as well as the novelty of our paper are subsequently illustrated through the comparison with the previous works.\\[-10pt]

\textit{Previous results for ``well-posedness"}:
As for the VFP-INS system, based on the hypoelliptic structure, Goudon-He-Moussa-Zhang proved the global-in-time existence of classical solutions and long-time behavior in the case of initial data close to equilibrium in \cite{GHMAZP10}. Chae-Kang-Lee showed the existence of weak solutions both in 2D and 3D as well as the existence and uniqueness of the smooth solutions in 2D in \cite{CMKKLJ11}. Moreover, as for the nonlinear VFP equation coupled with the Navier-Stokes equation, Carrillo-Choi-Karper proved the global existence and long-time behavior of weak solutions in \cite{CJACYPK16}.
Beyond the Cauchy problem, substantial advancements have been made in addressing the initial boundary value problem of the VFP equation coupled with the incompressible Navier-Stokes equation, where the existence of weak solution was proved by Yu in \cite{YC13} and Yu-Wang in \cite{WDHYC15} with the specular boundary condition; more recently, the global-in-time existence of the strong solution and its exponential convergence to the equilibrium has been derived by Li-Liu-Yang in \cite{LHLLSQYT22}, where the Maxwell boundary condition was considered.
For the VFP equation coupled with other fluid systems, the global existence of weak and strong solutions to the initial value problem for the VFP equation coupled with the compressible Navier-Stokes equation have been studied in \cite{CG06, CYP17, LFCMYMWDH17, MV07, MYMWDH20}, while the global-in-time existence and long-time behavior of the VFP equation coupled with the Euler equation can be found in \cite{CJADRJMA11, DL13}.
In addition to the VFP equation, various other forms of kinetic equations can be intricately coupled with the fluid equation: for instance, Jabin-Soler studied the coupled Boltzmann and Navier-Stokes fragmentation model by a refined averaging lemma \cite{Jabin2010}, and we refer the readers to \cite{CKL19,LYZ20-2, LYZ21} for recent progress toward the Vlasov-Poisson-Boltzmann system; besides, the Cucker-Smale model has also attracted attention as a typical kinetic equation to describe the flocking behavior \cite{Karper2013, Peszek1, Peszek2}.\\[-10pt]


\textit{Previous results for ``hydrodynamic limit"}:
In the seminal work \cite{GTJPEVA04-1}, Goudon-Jabin-Vasseur established a systematic framework to study the hydrodynamic limit of VFP-INS system \eqref{VFP-INS}, where they referred to such a scaling for the so-called light particles (density of the particle is much less than the gas density, i.e., $\rho_p/\rho_g = O(\v^2)$, where $\rho_p$ is the density of one single particle (mass/volume), while $\rho_g$ is the density of the fluid modeled by $\eqref{VFP-INS}_2-\eqref{VFP-INS}_3$, which is supposed to remain constant for the incompressible fluid, see \cite[Eq.~1.2-1.3]{GTJPEVA04-1} for more details), as informed by dimensionless analysis; they worked on the weak solution to the scaled VFP-INS system \eqref{VFP-INS} in 2D case, and thanks to a\,\textit{priori} uniform estimates and compactness argument, the limiting density and velocity field, at the cost of subsequence, were proved to be the solution to the NS-KS system \eqref{NS-KS} in a weak sense.
In their follow-up work \cite{GTJPEVA04}, another different scaling was proposed for the fine particles (density of the particle is of the same order as gas density, i.e., $\rho_p \approx \rho_g$), the hydrodynamic limit of which was rigorously justified by the well-known relative entropy method, more precisely, it leads the weak solution $(f^\v, u^\v)$ of the scaled coupled system to the smooth solution of the limiting macroscopic system in the sense of relative entropy, which implies the $L^1$ and $L^2$ convergence, respectively.
As a powerful tool introduced by Yau \cite{HTYau1991} (also referring to the weak-strong uniqueness principle by Dafermos \cite{Dafermos1996} and Lions \cite{Lions1998}), the relative entropy method has also been widely used in the justification of hydrodynamic limit for the other kinetic-fluid coupled systems, in addition to the hydrodynamic limit of Boltzmann equation by Bardos-Golse-Levermore \cite{BCGFLCD93} and Golse-Saint-Raymond \cite{ FGLSR04}.
For example, Su-Yao extended the result \cite{GTJPEVA04} to the inhomogeneous VFP-INS system in the bounded domain in \cite{YFSLY20}, while for the VFP equation coupled with the compressible Navier-Stokes equation, Mellet-Vasseur in \cite{MV08} studied the hydrodynamic limit in a bounded domain with Dirichlet or reflection boundary conditions based on their previous work \cite{MV07}, whereas Choi-Jung further generalized their results to the nonlinear VFP-INS system in \cite{YPCJJ21}.
Furthermore, Choi and his collaborator also derived the hydrodynamic limit for the charged particle interacting with the incompressible viscous flow in \cite{Choi23-2} and a Kuramoto model for synchronization phenomena \cite{Choi23-1}.
When considering the collective flocking behavior of particles, the well-recognized kinetic Cucker-Smale equation emerges as the focal point, coupled with the fluid equation through the mediation of drag forces: this type of kinetic-fluid model was first proposed by Bae-Choi-Ha-Kang in \cite{Kang4} to describe the dynamics of self-propelled particles immersed by the incompressible viscous fluid, and related models had been further discussed in \cite{CJACYPK16, Choi23-3, Kang2, Kang3, Poyato2021}, including the rigorous derivation from the kinetic Cucker-Smale model to the pressureless Euler system by Figalli-Kang in \cite{Kang1}. The recent development of the hydrodynamics limit for the flocking model can be found in \cite{Tadmor1} by Tadmor. Additionally, in the case where the diffusion does not dominate the behavior, we refer the readers to the high-field limit work of Poupaud and his collaborators in \cite{Poupaud2, Poupaud3}. For a vast literature about the hydrodynamic limit, we refer readers to \cite{JB20, YPCJJ20, YPCBK21, TG01, LYZ20} and the references therein.\\[-10pt]

\textit{Mathematical challenges and our contributions}:  As discussed above, the existing result \cite{GTJPEVA04-1} concerning the hydrodynamic limit of the specifically scaled VFP-INS system \eqref{VFP-INS} is basically attributable to the employment of the uniform estimate with the convergence achieved in a weak sense. However, based on such a sort of (weak) compactness argument, it is extremely difficult to obtain an accurate pointwise convergence. Even if using the relative entropy method as presented in \cite{GTJPEVA04} for the same system with a different scaling, no pointwise convergence could be expected, since only the $L^1$ or $L^2$ convergence can be implied from the convergence in the sense of relative entropy. 

The main purpose of our paper is to rigorously prove the hydrodynamic limit from the VFP-INS system \eqref{VFP-INS} to the NS-KS system \eqref{NS-KS} in a stronger pointwise sense, and hope to obtain an explicit convergence rate simultaneously. To achieve this goal, we use another type of methodology inspired by Caflisch's work \cite{Caflisch} for studying the hydrodynamic limit of the Boltzmann equation. The main idea is to seek a particular form of solution to the scaled kinetic-fluid system (e.g., \eqref{VFP-INS} in our model) by the Hilbert expansion around the classical solution to the limiting macroscopic system (e.g., \eqref{NS-KS} in our model). Following this framework, one has to show the well-posedness of the limiting macroscopic system, and the next crucial step involves substantiating uniform energy estimates for the remainder system (e.g., \eqref{Remainder equations} in our model), which is typically less singular or nonlinear than the original scaled kinetic-fluid system. Therefore, the pointwise convergence becomes feasible as long as the high-regularity energy estimate can be closed. This approach has recently been used to study self-organized kinetic equations and the system coupled with Navier-Stokes equations by Jiang-Luo-Zhang in \cite{JLZ20, JNXLJZTF16}.



For this problem, however, if we attempt with the usual Hilbert expansion of $(f^\v, u^\v)$ for the scaled VFP-INS system \eqref{VFP-INS} as follows:
\begin{equation}\label{Special form solution for guess}
f^\v=f_0 + \v f , \quad u^\v = u_0 + \v u,
\end{equation}
and substitute \eqref{Special form solution for guess} into  $\eqref{VFP-INS}_1$, then we find that the equation of $f$ becomes
\begin{multline}\label{Guess of form solution}
    \p_t f-\frac{1}{\v^2}\up{div}_v(\n_v f+v f)=\\
    -\frac{1}{\v^2}\big[v\c\n_x f_0+\up{div}_v(u_0 f_0)\big]-\frac{1}{\v}\big[\p_tf_0+v\c\n_x f+\up{div}_v(uf_0+u_0f)\big]-\up{div}_v(uf),
\end{multline}
with $f_0=n_0 M$ (here $n_0$ is the solution to the limiting system \eqref{NS-KS}), which can be derived by formal analysis.
We observe that the dissipation on the left-hand side essentially comes from $-\frac{1}{\v^2}\up{div}_v(\n_v f+v f)$ with the order of $\frac{1}{\v^2}$. In this case, it would be extremely hard to use the dissipation to control the ``bad'' term $-\frac{1}{\v^2}\big[v\c\n_x f_0+\up{div}_v(u_0 f_0)\big]$, since this singular term has the same order of $\frac{1}{\v^2}$.

Hence, to overcome this difficulty, the significance of our contributions lies in designing a brand-new expansion form of $(f^\v, u^\v)$ and constructing the well-matched energy/dissipation structure accordingly, such that we can prove the well-posedness of $(f^\v, u^\v)$ based on a refined energy estimate (Theorem \ref{Main-Limits}); as a direct consequence, the pointwise convergence can be achieved with an explicit convergent rate (Corollary \ref{Conergence}) thanks to the embedding relation. More precisely,


\begin{itemize}

\item For the solution $(f^\v, u^\v)$ to the scaled VFP-INS system \eqref{VFP-INS}, we re-design a new expanded form of $(f^\v, u^\v)$ as follows:
\begin{equation}\label{new_expansion}
f^\v=g_0 + \v g_1\M + \v g\M, \quad u^\v =u_0+\v u,
\end{equation}
where $g_0=n_0M,\,g_1=\big[(v\c u_0)n_0-v\c\n_xn_0\big]\M$. Then the equation of remainder $g$ becomes
\begin{equation*}
\p_t g-\frac{1}{\v}v\c u\M+\frac{1}{\v^2}L g =\frac{1}{\v}R_0+R_1,
\end{equation*}
where the operator $L$ and $R_0,\,R_1$ are defined in \eqref{L-R0-R3}. In this case, the lower-order singular terms $\frac{1}{\v}R_0,\, R_1$ on the right-hand side turns out to be controlled by the dissipation induced by the linearized operator $L$ in \eqref{def-L} with the order of $\frac{1}{\v^2}$.
It is crucial to highlight that our newly-designed expansion is not born in the middle of nowhere, and actually, we are inspired by the formal derivation, which is specifically presented in Section \ref{sec:formal}. This methodology exhibits significant potential for broader applications in addressing other hydrodynamic limit problems, particularly when the classical expanded structure is not appropriate to complete the energy estimates.\\[1mm]

\item To achieve the pointwise convergence $(f^\v, u^\v)$, we construct the delicate energy/dissipation structures as in Section \ref{subsec:energy_dissipation} for the remainder function $(g,u)$ arising from our new expansion \eqref{new_expansion}. The energy/dissipation structures are constructed based on the well-known macro-micro decomposition, which perfectly matches our new expansion. Moreover, the energy/dissipation functionals also involve the high regularity space in both $x$ and $v$ variables such that the pointwise convergence can directly follow from the total energy estimate (Proposition \ref{total_energy_dissipation}) by the embedding theorem.

\end{itemize}


To our best knowledge, this is the first result to rigorously study the hydrodynamic limits, specifically from the scaled VFP-INS system \eqref{VFP-INS} to the NS-KS system \eqref{NS-KS}, in the stronger pointwise sense, where explicit convergence rate is achieved globally in time, thanks to a new expansion form and delicate high-order energy estimate. In our forthcoming work \cite{FangQiWen}, we will consider the compressible case. In contrast with the incompressible case where the density of the fluid is supposed to be constant, the appearances of the extra continuity equation and the density-dependent pressure function in the coupled compressible Navier-Stokes equations would bring new difficulties.


\subsection{Notation}
\label{subsec:notation}
Before stating the main result, we give some notations in this paper.
\begin{itemize}
\item $A \lesssim B \Leftrightarrow A \leq C B$, $A \sim B\Leftrightarrow C_1 A \leq B \leq C_2 A$, for some generic constants $C,C_1,C_2>0$ independent of $t$ and $\v$.

\item For multi-indices $\alpha = (\alpha_1, \alpha_2, \alpha_3)$ and $\beta = (\beta_1, \beta_2, \beta_3)$, we denote
\begin{equation*}
    \partial_x^\alpha = \partial_{x_1}^{\alpha_1} \partial_{x_2}^{\alpha_2} \partial_{x_3}^{\alpha_3}, \quad \partial_v^\beta = \partial_{v_1}^{\beta_1} \partial_{v_2}^{\beta_2} \partial_{v_3}^{\beta_3}.
\end{equation*}

\item $\binom{\a}{\a'}$ denotes the binomial coefficient.

\item Let $\nu=1+|v|^2$ and denote $\|\c\|_{\nu}$ by
\begin{equation*}
\|g\|_{\nu}=\big( \int_{\R^3}\int_{\R^3} |\n_v g|^2+|g|^2\nu \,dvdx \big)^{\frac{1}{2}}.
\end{equation*}

\item For $d,e\in\mathbb{N}$, we denote the following inner-product and the spaces:
\begin{equation*}
\begin{aligned}
\l u,w\r_x =& \int_{\R^3} uw \,dx,\,  \quad  \l f,g\r_v=\int_{\R^3} fg \,dv,\, \quad \|u\|_{L^2_x}= \l u,u\r_x^{\frac{1}{2}},\, \quad \|f\|_{L^2_v}= \l f,f\r_v^{\frac{1}{2}}, \\[4pt]
\l f,g \r_{x,v} =& \int_{\R^3} \int_{\R^3} fg \,dvdx,\, \quad \|f\|_{L^2_{x,v}}=\l f,f\r_{x,v}^{\frac{1}{2}},\\
\end{aligned}
\end{equation*}

\begin{equation*}
\begin{aligned}
H^d_{x}:=&\big\{u(x) \ \big| \ \|\p_x^\alpha u\|_{L^2_x} <\infty,\,\up{for any}\,|\a|\leq d\big\},\\[6pt]
H^d_{x,v}:=&\big\{f(x,v) \ \big| \ \|\p_x^\alpha\p_v^\beta f\|_{L^2_{x,v}} <\infty,\,\up{for any}\,|\a|+|\b|\leq d\big\},\\[6pt]
\mathcal{H}^d_{x,v}:=&\big\{f(x,v) \ \big| \ \|\p_x^\alpha\p_v^{\beta+1} f\|_{L^2_{x,v}} <\infty,\,\up{for any}\,|\a|+|\b|+1\leq d\big\},\\[6pt]
H^d_x H^e_v:=&\big\{f(x,v) \ \big| \ \|\p_x^\alpha\p_v^\beta f\|_{L^2_{x,v}} <\infty,\,\up{for any}\,|\a|\leq d,\,|\b|\leq e\big\},\\[6pt]
\mathcal{V}^d_x:=&\big\{u(x)\in H^d_x \ \big|\ \up{div}_x u=0\},
\end{aligned}
\end{equation*}
where the space $\mathcal{V}_x^d$ has the same norm as $H^d_x$.
\end{itemize}

\subsection{Main result}
\label{subsec:main}

Let $m_0:=n_0-1$, the NS-KS system \eqref{NS-KS} is rewritten as, for $(t,x)\in \R^+\times\R^3$,
\begin{equation}\label{NS-KS-1}
\left\{
\begin{aligned}
&\p_t m_0+\up{div}_x(u_0m_0)=\Delta_x m_0,\\[4pt]
&\p_t u_0+u_0\c\n_x u_0+\n_x p_0=\Delta_x u_0,\\[4pt]
&\up{div}_x u_0=0,\\[4pt]
&(m_0,u_0)|_{t=0}=(m_0^{in}(x), u^{in}_0(x)),\\[4pt]
&(m_0(t,x),u_0(t,x))\to (0,0),\quad \up{as}\, |x|\to +\infty.
\end{aligned}
\right.
\end{equation}

Here, we provide a direct decoupling strategy to show the well-posedness of the NS-KS system \eqref{NS-KS-1}: we can first solve the following Navier-Stokes (NS) equation
\begin{equation}\label{NS}
\left\{
\begin{aligned}
&\p_t u_0+u_0\c\n_x u_0+\n_x p_0=\Delta_x u_0,\\[4pt]
&\up{div}_x u_0=0,\\[4pt]
&u_0(0,x)=u^{in}_0(x),\\[4pt]
&u_0(t,x)\to 0, \quad \up{as} \quad |x|\to +\infty,
\end{aligned}
\right.
\end{equation}
to obtain the solution $(u_0,p_0)$; with the known $u_0$, we then solve the Kramer-Smoluchowski (KS) equation
\begin{equation}\label{KS equation}
\left\{
\begin{aligned}
&\p_t m_0+\up{div}_x(m_0 u_0)=\Delta_x m_0,\\[4pt]
&m_0(0,x)=m^{in}_0(x),\\[4pt]
&m_0(t,x)\to 0, \quad \up{as} \quad |x|\to +\infty,
\end{aligned}
\right.
\end{equation}
to get $m_0$. Finally, the solution $(m_0,u_0,p_0)$ to the NS-KS system \eqref{NS-KS-1} can be obtained.

\begin{proposition}\label{Solu-NS-KS}
Assume that the initial data $(m_0^{in}, u_0^{in})\in H^6_x \times \mathcal{V}_x^6$,
and there exists a small constant $\delta > 0$ such that, if $\|(m_0^{in},u_0^{in})\|_{H^6_x} \leq \delta$, then the system \eqref{NS-KS-1} admits the unique global-in-time solution $(m_0, u_0)$ satisfying
\begin{equation}
m_0 \in C^0 \left( [0,+\infty); H^6_x \right) \, \cap \, C^1 \left( [0,+\infty); H^4_x \right), \quad
u_0\in C^0 \left( [0,+\infty);\mathcal{V}_x^6 \right) \,\cap\, C^1\left( [0,+\infty);\mathcal{V}_x^4 \right).
\end{equation}
and, for any $0 < T < \infty$,
\begin{equation}\label{H6}
    \H(T;6) := \sup_{0\leq t\leq T}\|(m_0(t,\cdot),u_0(t,\cdot))\|^2_{H^6_x} \lesssim \delta.   
\end{equation}
\end{proposition}

\begin{proof}
In fact, the solvability of the incompressible NS system \eqref{NS} can be referred to the argument as that in \cite{AJMALB02}, which is well-known for the initial data $u_0^{in}$ with ``smallness" assumption. Then, the existence of the solution $m_0$ to (\ref{KS equation}) can be derived by using the standard parabolic theory. 
\end{proof}


Based on $(m_0,u_0)$ of the NS-KS system \eqref{NS-KS-1}, we seek a particular solution to the VFP-INS system \eqref{VFP-INS} in the following form:
\begin{equation}\label{Special form solution}
\begin{aligned}
f^\v= &\ (1+m_0)M + \v[(v\c u_0)(1+m_0)-v\c\n_x m_0]M + \v g\M,\\[4pt]
u^\v= &\ u_0+\v u.
\end{aligned}
\end{equation}
which is inspired by the formal derivation provided in Section \ref{sec:formal}.



Now, we are in a position to state our main theorem and the complete proof can be found in Section \ref{subsec:proof_main_theorem}.

\begin{theorem}\label{Main-Limits}
Let $(m_0,u_0)$ be the solution to the NS-KS system \eqref{NS-KS-1} provided by Proposition \ref{Solu-NS-KS} with the initial data $(m_0^{in},u_0^{in}) \in H^6_x\times\mathcal{V}^6_x$, and assume that the initial data of \eqref{VFP-INS} are given by
\begin{equation}\label{Theorem-intial data of VFP-NS}
\begin{aligned}
f^{\v,in}(x,v)=&\ (1+m_0^{in})M+\v\big[(v\c u_0^{in})(1+m^{in}_0)-v\c\n_x m^{in}_0\big]M+\v g^{in}\M,\\[4pt]
u^{\v,in}(x)=&\ u_0^{in}+\v u^{in},
\end{aligned}
\end{equation} for some known functions $g^{in}\in  H^4_{x,v},\,
u^{in}\in  \mathcal{V}^4_x.$

Then there exist $g(t,x,v)$, $u(t,x)$, and a small constant $\delta_0 > 0$ such that, for any $ 0 < \v \leq 1$ and $\mathbb{E}(0) \leq \delta_0$, the VFP-INS system \eqref{VFP-INS} admits a unique global-in-time solution $(f^\v(t,x,v),u^\v(t,x))$ in the following form:
\begin{equation}\label{Form of the solution}
\begin{aligned}
f^\v(t,x,v)=&\ (1+m_0)M+\v\big[(v\c u_0)(1+m_0)-v\c\n_x m_0\big]M+\v g\M,\\[4pt]
u^\v(t,x)=&\ u_0+\v u,
\end{aligned}
\end{equation}
where
\begin{equation}\label{Space of the solution}
\begin{aligned}
&g \in C^0 \left([0,\infty);H^4_{x,v}\right) \cap L^2 \left([0,\infty);\mathcal{H}^5_{x,v}\right), \, \sqrt{M}g\in C^0 \left([0,\infty);H^2_{x}H^2_v\right),\\
&u \in C^0 \left([0,\infty);\mathcal{V}^4_x\right) \cap L^2\left([0,\infty);\mathcal{V}^5_x\right),
\end{aligned}
\end{equation}
and the global-in-time energy estimate:
\begin{equation}\label{The global estimate}
\sup_{t\geq0} \mathbb{E}(t)+c_0\int_0^\infty \mathbb{D}(s) \,ds \lesssim \mathbb{E}(0),
\end{equation} holds, where $c_0 > 0$ is a constant independent of $\v$, and the energy functional $\mathbb{E}(t)$ and dissipation functional $\mathbb{D}(t)$ are defined in \eqref{Def_total_functional}.
\end{theorem}


Thanks to the Theorem \ref{Main-Limits}, the following corollary can be obtained by applying the Sobolev embedding $H^2\hookrightarrow L^\infty$, and the complete proof is given in Section \ref{subsec:proof_main_corollary}.

\begin{corollary}\label{Conergence}
Under the conditions of Theorem \ref{Main-Limits}, the following pointwise convergence:
\begin{equation}\label{Conv of t,x}
\begin{aligned}
|f^\v(t,x,v)-(1+m_0(t,x))M(v)|+|u^\v(t,x)-u_0(t,x)|\lesssim \v
\end{aligned}
\end{equation} holds for all $(t,x,v)\in [0,\infty) \times\R^3\times\R^3$.
\end{corollary}


\subsection{Organization of the paper}
\label{subsec:organization}
This paper is organized as follows: In Section \ref{sec:formal}, the formal derivation is shown via the Hilbert expansion. In Section \ref{sec:energy}, we develop a global-in-time energy estimate for the remainder/error system. Based on the delicate energy estimate above, our main results are proved in Section \ref{sec:proof_main}.





\section{Formal analysis via Hilbert expansion}
\label{sec:formal}

In this section, we first apply the Hilbert expansion to make the formal derivation:
assuming
\begin{equation}\label{Formal form}
\begin{aligned}
&f^\v=f_0+\v f_1+\v^2 f_2+\cdots,\\
&u^\v=u_0+\v u_1+\v^2 u_2+\cdots,
\end{aligned}
\end{equation}
and plugging the expansion \eqref{Formal form} into the VFP-INS system $\eqref{VFP-INS}_1$, we obtain

\begin{equation}\label{Formally form about f}
\begin{aligned}
&\p_t(f_0+\v f_1+\v^2 f_2+\cdots)+\frac{1}{\v}v\c\n_x(f_0+\v f_1+\v^2 f_2+\cdots )\\
=&\frac{1}{\v^2}\up{div}_v\Big\{\n_v(f_0+\v f_1+\v^2 f_2+\cdots)+\big[v-\v(u_0+\v u_1+\v^2 u_2+\cdots)\big]  (f_0+\v f_1+\v^2 f_2+\cdots)\Big\}.
\end{aligned}
\end{equation}

For $O(\frac{1}{\v^2})$ term of \eqref{Formally form about f}, we obtain the equation only involving $f_0$,
\begin{equation*}
    \up{div}_v(\n_v f_0+v f_0)=\up{div}_v[M\n_v(\frac{f_0}{M})]=0,
\end{equation*}
which implies that
\begin{equation}\label{macar-n0}
f_0=n_0 M,
\end{equation}
for some functions $n_0(t,x)$.

For $O(\frac{1}{\v})$ of \eqref{Formally form about f}, we find
\begin{equation}\label{macar-1}
v\c\n_x f_0+\up{div}_v(u_0 f_0)=\up{div}_v[M\n_v(\frac{f_1}{M})].
\end{equation}
By further substituting \eqref{macar-n0} into \eqref{macar-1} and making simplification, we have
\begin{equation*}
\up{div}_v\big[M(\n_v\frac{f_1}{M}+\n_x n_0-u_0 n_0)\big]=0,
\end{equation*}
from which, we can solve $f_1$ in the following form
\begin{equation}\label{Solu-f1}
f_1=\big[(v\c u_0)n_0-v\c\n_x n_0\big]M.
\end{equation}

For the $O(1)$ of \eqref{Formally form about f}, we have
\begin{equation*}
\p_t f_0+v\c\n_x f_1=\up{div}_v(\n_v f_2+vf_2-u_0f_1-u_1f_0),
\end{equation*}
which implies that
\begin{equation}\label{Equa of n0}
\int_{\R^3} \p_t f_0+v\c\n_x f_1 \,dv=0.
\end{equation}

Finally, substituting \eqref{macar-n0} and \eqref{Solu-f1} into \eqref{Equa of n0} gives
\begin{equation}\label{n0}
\p_t n_0+\up{div}_x(n_0 u_0)=\Delta_x n_0.
\end{equation}

On the other hand, substituting the expansion \eqref{Formal form} into the VFP-INS system $\eqref{VFP-INS}_2$, it follows that
\begin{multline}\label{Formal VFP-2}
\p_t(u_0+\v u_1+\cdots)+\mathbb{P}\big[(u_0+\v u_1+\cdots)\c\n_x(u_0+\v u_1+\cdots)\big]\\
=\Delta_x(u_0+\v u_1+\cdots)+\frac{1}{\v}\mathbb{P}\Big\{\int_{\R^3}\big[v-\v(u_0+\v u_1+\cdots)\big](f_0+\v f_1+\cdots)dv\Big\},
\end{multline}
where the symbol $\mathbb{P}$ is the Leray projection on $\R^3$.

The $O(\frac{1}{\v})$ term of \eqref{Formal VFP-2} reads
\begin{equation*}
\mathbb{P}\big(\int_{\R^3} vf_0 \,dv \big)=0.
\end{equation*}
which automatically holds if we consider $f_0=n_0(t,x)M$ in \eqref{macar-n0}.

The $O(1)$ term of \eqref{Formal VFP-2} reads
\begin{equation*}
\p_t u_0+u_0\c\n_x u_0-\Delta_xu_0 =\mathbb{P}\big(\int_{\R^3}v f_1 -u_0 f_0 \,dv \big),
\end{equation*}
which implies that
\begin{equation}\label{u0}
\p_t u_0+u_0\c\n_x u_0+\n_x (p_0+n_0)-\Delta_xu_0=0,
\end{equation}
since $f_0=n_0 M$ and $f_1=\big[(v\c u_0)n_0-v\c\n_x n_0\big]M$.

At last, substituting the expansions \eqref{Formal form} into the VFP-INS system $\eqref{VFP-INS}_3$ gives
\begin{equation}\label{Divergence free of u}
\up{div}_x u_j=0, \quad j = 0, 1, 2,...
\end{equation}
which implies the incompressible condition of $u_j$ for each order.

Thus, combining \eqref{n0}, \eqref{u0} and \eqref{Divergence free of u} gives the NS-KS system \eqref{NS-KS} in a formal manner.


\section{Energy estimate}
\label{sec:energy}

The essence of the proof for our main theorems is the global-in-time energy estimate uniform in $0 < \v \leq 1$ for the small initial data, i.e., the estimate \eqref{The global estimate}. To achieve this, our proof of \eqref{The global estimate} can be outlined as follows: we first present the remainder/error system of $(g,u)$ in subsection \ref{subsec:remainder}, and the specific designs of corresponding energy and dissipation functionals are shown in subsection \ref{subsec:energy_dissipation}; as the whole principle of the energy estimate is trying to find sufficient dissipative or decay structures to control the singularity terms, we have to rely on the estimate of the NS-KS system \eqref{NS-KS-1} in subsection \ref{subsec:apriori_NS-KS} and capture such ``good" dissipative structure from both microscopic part and macroscopic part (micro-macro decomposition) of the reminder system in subsection \ref{subsec:remainder_kinetic} and \ref{subsec:remainder_fluid}, respectively; we finally summarize all the estimates and conclude the total energy estimate in a well-designed and closed manner in subsection \ref{subsec:proof_total_energy}.

\subsection{Micro-macro decomposition}
Applying micro-macro decomposition in \cite{DFT10}, we decompose $g$ as the sum of the macroscopic part $\P g$ and the microscopic part $(\I-\P)g$ as follows:
\begin{equation}
g=\P g+(\I-\P)g.
\end{equation}
The velocity orthogonal projection $\P\up{:}\,L^2_v\to\up{Span}\{\sqrt{M},v_1\sqrt{M},v_2\sqrt{M},v_3\sqrt{M}\}$ is given by
\begin{equation}\label{def-L-P}
\begin{aligned}
&\P=\P_0\oplus \P_1,\, \quad \,\P_0 g=a\M,\, \quad \,\P_1 g=v\c b\M
\end{aligned}
\end{equation}
with
\begin{equation}\label{ab}
    a=\int_{\R^3} g \M \,dv \quad \text{and} \quad b=\int_{\R^3} gv \M \,dv.
\end{equation}

The linearized Fokker-Planck operator $L$ is defined as
\begin{equation}\label{def-L}
L g=-\frac{1}{\M}\up{div}_v\big[M\n_v(\frac{g}{\M})\big],
\end{equation}
then, the linear operator $L $ is decomposed by
\begin{equation}\label{Decop-L}
L g=L(\I-\P)g+\P_1g=L(\I-\P)g+v\c b\M.
\end{equation}

Notice that $\I-\P_0$, $\I-\P$ is self-adjoint in $H^d_{x,v}$, i.e.
\begin{equation}\label{self-adjoin of I-P}
\begin{aligned}
\l\p_x^\alpha(\I-\P_0)f,\p_x^\alpha g\r_{x,v} = \l\p_x^\alpha f,\p_x^\alpha(\I-\P_0)g\r_{x,v}, \quad
\l\p_x^\alpha(\I-\P)f,\p_x^\alpha g\r_{x,v} = &\l\p_x^\alpha f,\p_x^\alpha(\I-\P)g\r_{x,v},\\
\end{aligned}
\end{equation}
for any $\alpha$ and $f,g\in H^d_{x,v}$, and it is easy to verify that
\begin{equation}\label{I-P}
\begin{aligned}
(\I-\P_0)(\I-\P)=(\I-\P),\quad (\I-\P_0)(\I-\P_0)=(\I-\P_0),\quad (\I-\P)(\I-\P)=(\I-\P).
\end{aligned}
\end{equation}

According to \cite{CJADRJMA11}, the operator $L$ enjoys the dissipative property, i.e., there exists a constant $C_0>0$ such that
\begin{equation}\label{Poin inequality}
C_0\|(\I-\P)f\|_{\nu}^2+\|b\|^2_{L^2_x}  \leq \l L f, f\r_{x,v}.
\end{equation}

\subsection{The remainder/error system}
\label{subsec:remainder}

Based on the formal analysis, we introduce remainder/error terms of the VFP-INS system \eqref{VFP-INS} as follows:
\begin{equation}\label{Solu of special form-0}
g=\frac{f^\v-g_0-\v g_1\M}{\v \M}, \quad u=\frac{u^\v-u_0}{\v},
\end{equation}
in other words, we seek a solution to VFP-INS system \eqref{VFP-INS} in a special form:
\begin{equation}\label{Solu of special form-1}
f^\v=g_0+\v g_1 \M+\v g\M,  \quad u^\v=u_0+\v u,
\end{equation}
where $g_0=n_0M,\,g_1=\big[(v\c u_0)n_0-v\c\n_x n_0\big]\M$ and $(n_0,u_0)$ is the solution to the NS-KS system \eqref{NS-KS}.

Noting $m_0=n_0-1$ and substituting the expansion \eqref{Solu of special form-1} into \eqref{VFP-INS}, we can obtain the remainder system of $(g,u)$:
\begin{equation}\label{Remainder equations}
\left\{
\begin{aligned}
& \p_t g-\frac{1}{\v}v\c u\M+\frac{1}{\v^2}L g=\frac{1}{\v}R_0+R_1,\\
& \p_t u+u_0\c\n_x u+u\c\n_x u_0-\frac{1}{\v}(b-\v u)+\frac{1}{\v}\n_x(p-p_0) - \Delta u =R_2,\\
& \up{div}_x u=0,\\
& g(0,x,v)=g^{in}(x,v), \quad u(0,x)=u^{in}(x),\\
& \big( g(t,x,v),u(t,x) \big)\to (0,0),\,\, \up{as}\, |x|\to +\infty,
\end{aligned}
\right.
\end{equation}
where $(m_0,u_0)$ is the solution to the NS-KS system \eqref{NS-KS-1} as in Proposition \ref{Solu-NS-KS} and
\begin{equation}\label{L-R0-R3}
\begin{aligned}
L g=&-\frac{1}{\M}\up{div}_v[M\n_v(\frac{g}{\M})]=-\Delta_v g+\frac{|v|^2}{4}g-\frac{3}{2}g,\\[2pt]
R_0=&-(\I-\P_0)(v\c\n_x g_1)-v\c\n_x g-\frac{1}{\M}\up{div}_v(u_0 g_1\M+u_0 g\M+um_0 M),\\[2pt]
R_1=&-\p_t g_1-\frac{1}{\M}\up{div}_v(u g_1\M+ug\M),\\[2pt]
R_2=&-(u_0 a+u m_0+\v ua),\\[2pt]
g_1=&\big[(v\c u_0)m_0+v\c u_0-v\c\n_x m_0\big]\M.
\end{aligned}
\end{equation}

For the remainder system \eqref{Remainder equations} above, we first give the local well-posed result:

\begin{proposition}\label{Local-in-time}
Under the conditions of Proposition \ref{Solu-NS-KS} and the initial data $(g^{in},u^{in} ) \in H^4_{x,v} \times \mathcal{V}^4_x$, for any $0<\v\leq 1$,
there exists a $T_\v > 0$ such that the remainder system \eqref{Remainder equations} admits a unique solution $(g,u)$ with
$$g\in C^0 \left([0,T_\v];H^4_{x,v}\right)\,\cap\, L^2\left([0,T_\v];\mathcal{H}^5_{x,v}\right) \quad \text{and} \quad
u\in C^0 \left([0,T_\v];\mathcal{V}^4_x\right) \,\cap\, L^2\left([0,T_\v];\mathcal{V}^5_x\right).$$
\end{proposition}

\begin{proof}
The proof is based on the standard contractive mapping argument. We refer to \cite{GHMAZP10} for more details.
\end{proof}

\subsection{Total energy estimate}
\label{subsec:energy_dissipation}

To state the energy estimate, we will introduce the temporal energy and dissipation functionals of different parts:
\begin{itemize}

\item Energy and the dissipation functionals of macroscopic quantities $(m_0, u_0)$:
\begin{equation}\label{macro energy functionals and energy dissipation functionals}
\begin{aligned}
\E_{ma}(t)=&\|m_0\|^2_{H^6_x}+\|u_0\|^2_{H_x^5},\\[4pt]
\D_{ma}(t)=&\|(\n_x m_0,\n_x u_0)\|^2_{H^5_x}+\|(\p_t\n_x m_0,\p_t u_0)\|_{H^4_x}^2.
\end{aligned}
\end{equation}

\item Energy and the dissipation functionals of remainder/error $(g,u)$:
\begin{equation}\label{part of energy functionals}
\begin{aligned}
\E_{mi,K,1}(t)=&\| g\|^2_{H^4_x L_v^2}+\|u\|^2_{H^4_x},\\[4pt]
\E_{mi,K,2}(t)=&\sum_{|\a|+|\b|\leq 3}\bar{C}_{\a,\b}\|\p_x^\a\p_v^{\b+\b'}(\I-\P)g\|^2_{L^2_{x,v}},\\[4pt]
\E_{mi,F}(t)=&\|(a,b)\|_{H^3_x}^2+2\sum_{|\a|=0}^3\sum_{i,j=1}^3\l\p_x^{\a}(\p_{x_i} b_j+\p_{x_j} b_i),\p_x^\a(\I-\P)g(v_iv_j-1)\M\r_{x,v}+\v\sum_{|\a|=0}^3\l\p_x^{\a+\a'}a,\p_x^\a b\r_x,\\[4pt]
\D_{mi,K,1}(t)=&\frac{1}{\v^2}\big[\sum_{|\a|\leq 4}\|\p_x^\a(\I-\P)g\|_{\nu}^2+\|b-\v u\|_{H^4_x}^2\big]+\|\n_x u\|_{H^4_x}^2,\\[4pt]
\D_{mi,K,2}(t)=&\frac{1}{\v^2}\sum_{|\a|+|\b|\leq 3}\|\p_x^\a\p_v^{\b+\b'}(\I-\P)g\|_{\nu}^2,\\[4pt]
\D_{mi,F}(t)=&\frac{1}{\v}\|(\n_x b,\up{div}_x b)\|^2_{H^3_x}+\|\n_x a\|^2_{H^3_x},\\[4pt]
\end{aligned}
\end{equation}
where $\bar{C}_{\a,\b} > 0$ are constants and $|\a'|=|\b'|=1$.

\item Total energy functional $\E(t)$ and dissipation functional $\D(t)$:
\begin{equation}\label{energy functionals and energy dissipation functional}
\begin{aligned}
\E(t)=&\lambda_1\E_{mi,F}(t)+\sum_{i=2}^3\lambda_i\E_{mi,K,i-1}(t)+\lambda_4\E_{ma}(t),\\[4pt]
\D(t)=&\D_{mi,F}(t)+\sum_{i=1}^2\D_{mi,K,i}(t)+\D_{ma}(t),
\end{aligned}
\end{equation}
where $\lambda_i > 0, 1\leq i\leq 4$ are constants given in \eqref{lambda constants}.

\end{itemize}

\begin{proposition}\label{total_energy_dissipation}
Let $(g,u)$ be the solution to the remainder system \eqref{Remainder equations}, there exists constants $C,\,\tilde{C} > 0$ independent of $\v$ and $t$ such that, for any $t \geq 0$,
\begin{equation}\label{total_energy_estimate}
\begin{aligned}
\frac{1}{2}\frac{d}{dt}\E(t)+\tilde{C}\D(t) \leq C\E^{\frac{1}{2}}(t)\D(t),
\end{aligned}
\end{equation}
where the energy and dissipation functionals $\E(t)$ and $\D(t)$ are defined in \eqref{energy functionals and energy dissipation functional}.
\end{proposition}


We also introduce another type of energy functional $\mathbb{E}(t)$ and dissipation functional $\mathbb{D}(t)$:
\begin{equation}\label{Def_total_functional}
\begin{aligned}
\mathbb{E}(t)= &\ \|g\|_{H^4_{x}L^2_v}^2+\|\n_v(\I-\P)g\|^2_{H^3_{x,v}}+\|(a,b)\|_{H^3_x}+\|u\|^2_{H^4_x}+\|m_0\|_{H^6_x}^2 + \|u_0\|_{H^5_x}^2,\\[6pt]
\mathbb{D}(t)= &\ \frac{1}{\v^2} \left(\sum_{|\a|+|\b|\leq 3}\|\p_x^\a\p_v^{\b+\b'}(\I-\P)g\|_{\nu}^2 + \sum_{|\a|+|\b|\leq 4}\|\p_x^\a\p_v^{\b}(\I-\P)g\|_{\nu}^2+\|b-\v u\|_{H^4_x}^2 \right) + \|\n_x u\|_{H^4_x}^2\\
&+\frac{1}{\v}\|(\n_x b,\up{div}_x b)\|^2_{H^3_x} + \|\n_x a\|^2_{H^3_x}+\|(\n_x m_0,\n_x u_0)\|^2_{H^5_x}+\|(\p_t\n_x m_0,\p_t u_0)\|_{H^4_x}^2
\end{aligned}
\end{equation}
with $|\b'|=1$.

\begin{remark}
By noticing the definitions of $\E(t),\,\D(t)$ in \eqref{energy functionals and energy dissipation functional}, we can directly check the following equivalent relation:
\begin{equation}
\begin{aligned}
\mathbb{E}(t)\sim \E(t),  \quad \mathbb{D}(t)\sim\D(t).
\end{aligned}
\end{equation}
\end{remark}

Based on Proposition \ref{total_energy_dissipation} and equivalent relation, we can directly obtain
\begin{corollary}\label{Instant energy-dissipative rate}
Let $(g,u)$ be the solution to the remainder system \eqref{Remainder equations}, it holds that, for any $\tau\in [0,\infty)$,
\begin{equation}\label{instant energy functional-energy dissipative rate functional}
\frac{1}{2} \mathbb{E}(\tau) + \int_0^\tau\mathbb{D}(s) \,ds \leq C\sup_{0 \leq s\leq \tau} \mathbb{E}^{\frac{1}{2}}(s)\int_0^\tau\mathbb{D}(s) \,ds + C\mathbb{E}(0),
\end{equation}
where the constant $C > 0$ is independent of $\v$ and $\tau$.
\end{corollary}


\subsection{\textit{A priori} estimates for the NS-KS system \eqref{NS-KS-1}}
\label{subsec:apriori_NS-KS}

In this section, we present \textit{a priori} estimates for the NS-KS system \eqref{NS-KS-1}. In fact, if we hope to have a good control of $\p_t g_1$ in \eqref{L-R0-R3}, we have to estimate $\p_t u_0$ and $\p_t \n_x m_0$, which is the reason why the estimates on $(m_0, u_0)$ are necessary.

To this end, we present the following two lemmas of estimates on $(m_0, u_0)$:

\begin{lemma}\label{Hs-estimate}
Let $(m_0,u_0)$ be the global-in-time solution to the NS-KS system \eqref{NS-KS-1}, then we have
\begin{equation}\label{Estimate-m-u}
\begin{aligned}
\frac{1}{2}\frac{d}{dt}\|
(m_0,u_0)\|^2_{H^6_x}+\frac{3}{4}\|(\n_x m_0, \n_x u_0)\|^2_{H^6_x}\leq 0
\end{aligned}
\end{equation}
for any $t \geq 0$.
\end{lemma}
\begin{proof}
By applying the derivative operator $\p_x^{\a}$ to $\eqref{NS-KS-1}$ for $0 \leq |\a| \leq 6$, taking inner product with $\p_x^{\a}u_0$, $\p_x^{\a}m_0$, and integrating by parts, respectively, we have, for $|\a'| = 1$,
\begin{equation}
\begin{aligned}
&\frac{1}{2}\frac{d}{dt}\|(\p_x^\a m_0, \p_x^\a u_0)\|_{L^2_x}^2 + \|(\p_x^{\a+\a'}m_0, \p_x^{\a+\a'}u_0)\|_{L^2_x}^2\\[4pt]
= \ &-\l\p_x^\a(u_0\c\n_x m_0),\p_x^\a m_0\r_x-\l\p_x^\a(u_0\c\n_x u_0),\p_x^\a u_0\r_x\\[4pt]
\leq\  & \sum_{0\leq |\tilde{\a}| \leq |\a|} \binom{\a}{\tilde{\a}}{}\Big(\|\p_x^{\tilde{\a}}u_0\|_{L^3_x}\|\p_x^{\a'+\a-\tilde{\a}}m_0\|_{L^2_x}\|\p_x^\a u_0\|_{L^6_x}+\|\p_x^{\tilde{\a}}u_0\|_{L^3_x}\|\p_x^{\a'+\a-\tilde{\a}}u_0\|_{L^2_x}\|\p_x^\a u_0\|_{L^6_x}\Big) \\[4pt]
\leq \ &  C\|(m_0, u_0)\|_{H^6_x}\|(\n_x m_0, \n_x u_0)\|_{H^6_x}^2\\[4pt]
\leq \ & C\sqrt{\delta}\|(\n_x m_0, \n_x u_0)\|_{H^6_x}^2,
\end{aligned}
\end{equation}
where we utilize H$\ddot{\up{o}}$lder inequality and Young inequality in the first inequality, the Sobolev embedding in the second inequality and \eqref{H6} in the last inequality.

The proof can be completed by summing up $0 \leq |\a| \leq 6$ and choosing $\delta>0$ small enough such that $C\sqrt{\delta}\leq\frac{1}{4}$.
\end{proof}

\begin{lemma}\label{partial-m0-u0}
Let $(m_0,u_0)$ be the global-in-time solution to the NS-KS system \eqref{NS-KS-1}, then for any $0\leq t\,<\infty$,
\begin{equation}\label{Hs-m0}
\begin{aligned}
\frac{1}{2}\frac{d}{dt} \|\p_x^{\a+\a''}m_0\|^2_{L^2_x} + \frac{3}{4} \|\p_t\p_x^{\a+\a'}m_0\|^2_{L^2_x} \leq \frac{1}{4} \|(\n_x m_0,\n_x u_0)\|^2_{H^{|\a|+1}_x},
\end{aligned}
\end{equation}
and
\begin{equation}\label{Hs-u0}
\begin{aligned}
&\frac{1}{2}\frac{d}{dt} \|\p_x^{\a+\a'}u_0\|^2_{L^2_x} + \frac{3}{4}\|\p_t\p_x^\a u_0\|^2_{L^2_x} \leq \frac{1}{4} \|\n_x u_0\|_{H^{|\a|}_x}^2,
\end{aligned}
\end{equation}
for any $0\leq|\a|\leq 4$, $|\a'|=1$, and $|\a''|=2$.
\end{lemma}
\begin{proof}

By applying the derivative operator $\p_x^{\a+\a'}$ to $\eqref{NS-KS-1}_1$ for $0\leq|\a|\leq 4$ and $|\a'|=1$, taking inner product with $\p_t\p_x^{\a+\a'}m_0$, and integrating by parts, we have, for $|\a''|=2$,
\begin{equation}
\begin{aligned}
&\frac{1}{2}\frac{d}{dt} \|\p_x^{\a+\a''}m_0\|^2_{L^2_x} +\|\p_t\p_x^{\a+\a'}m_0\|^2_{L^2_x} \\[4pt]
= &\ -\l\p_x^{\a+\a'}\up{div}_x(u_0m_0),\p_t\p_x^{\a+\a'}m_0\r_x\\[4pt]
\leq &\ \|m_0\|_{L^\infty_x} \|\p_x^{\a+\a'}\up{div}_xu_0\|_{L^2_x} \|\p_t\p_x^{\a+\a'}m_0\|_{L^2_x} +\|u_0\|_{L^\infty_x}\|\p_x^{\a+\a''}m_0\|_{L^2_x} \|\p_t\p_x^{\a+\a'}m_0\|_{L^2_x} \\[4pt]
&\ + C\sum_{0\leq|\tilde{\a}| \leq |\a|} \|\p_x^{\tilde{\a}}\up{div}_x u_0\|_{L^4_x} \|\p_x^{\a+\a'-\tilde{\a}}m_0\|_{L^4_x} \|\p_t\p_x^{\a+\a'}m_0\|_{L^2_x} \\
&\ +C\sum_{1 \leq |\tilde{\a}| \leq |\a|+|\a'|}\|\p_x^{\tilde{\a}} u_0\|_{L^4_x}\|\p_x^{\a+\a''-\tilde{\a}}m_0\|_{L^4_x}\|\p_t\p_x^{\a+\a'}m_0\|_{L^2_x} \\
\leq &\ C\|(m_0,u_0)\|_{H^5_x} \Big(\|\p_t\p_x^{\a+\a'}m_0\|^2_{L^2_x} + \|(\n_x u_0,\n_x m_0)\|^2_{H^{|\a|}_x}\Big)\\
\leq&\ \frac{1}{4} \|\p_t\p_x^{\a+\a'}m_0\|^2_{L^2_x} + \frac{1}{4}\|(\n_x m_0,\n_x u_0)\|^2_{H^{|\a|+1}_x},
\end{aligned}
\end{equation}
where we notice \eqref{H6} and choose $\delta$ small enough such that and $C\sqrt{\delta}\leq \frac{1}{4}$.

Similarly, after applying $\p_x^\a$ to $\eqref{NS-KS-1}_2$ for  $0\leq|\a|\leq 4$, taking the inner product with $\p_t\p_x^\a u_0$, and integrating by parts, we obtain, for $|\a'|=1$,
\begin{equation*}
\begin{aligned}
\frac{1}{2}\frac{d}{dt}\|\p_x^{\a+\a'}u_0\|^2_{L^2_x} +\|\p_t\p_x^\a u_0\|^2_{L^2_x} = &\ -\l\p_x^\a(u_0 \cdot \n_x u_0),\p_t\p_x^\a u_0\r_x\\
\leq& \ \frac{1}{4}\|\p_t\p_x^\a u_0\|^2_{L^2_x} + \frac{1}{4}\|\n_x u_0\|^2_{H^{|\a|}_x}.
\end{aligned}
\end{equation*}

\end{proof}



By using the Lemma \ref{Hs-estimate}, we have the following estimate of the macroscopic energy $\E_{ma}(t)$ and dissipation $\D_{ma}(t)$:

\begin{lemma}\label{Estimate of macro}
Let $(m_0,u_0)$ be the global-in-time solution to the NS-KS system \eqref{NS-KS-1}, we have
\begin{equation}\label{Macro-part}
\frac{1}{2} \frac{d}{dt} \E_{ma}(t) + \frac{1}{2} \D_{ma}(t) \leq 0,
\end{equation}
for any $t \geq 0$.
\end{lemma}

\begin{proof}
The estimate \eqref{Macro-part} is obtained by combining the inequalities \eqref{Estimate-m-u}, \eqref{Hs-m0}, \eqref{Hs-u0}, and Lemma \ref{Hs-estimate}.
\end{proof}


\subsection{\textit{A priori} estimates for the kinetic part of remainder system \eqref{Remainder equations}}
\label{subsec:remainder_kinetic}

The coercivity property of $L$ in the remainder system  \eqref{Remainder equations} produces the kinetic energy dissipation of $g$.

\begin{lemma}\label{Estimate-kinetic-Mi-1}
Let $(g,u)$ be the solution to the remainder system \eqref{Remainder equations}, then there exist constants $C_1,\,C_{1,1} > 0$ independent of $\v$ such that, for $t \geq 0$,
\begin{equation}\label{Step one}
\begin{aligned}
\frac{1}{2}\frac{d}{dt}\E_{mi,K,1}(t)+C_{1,1}\D_{mi,K,1}(t) \leq C_1 \big(\D_{ma}(t)+\E^{\frac{1}{2}}(t)\D(t)\big)
\end{aligned}
\end{equation}
where $\E_{mi,K,1}$, $\D_{mi,K,1}(t)$, $\E(t)$, and $\D(t)$ are defined in \eqref{part of energy functionals} and \eqref{energy functionals and energy dissipation functional}, respectively.
\end{lemma}

\begin{proof}
Applying the derivative operator $\p_x^\a$ with $0\leq|\a|\leq 4$ to the remainder system \eqref{Remainder equations}, taking $L^2_{x,v}$ and $L^2_x$-inner product with $\p_x^\a g$, $\p_x^\a u$, and then integrating by parts, we have, for $|\a'|=1$,
\begin{equation}\label{Estimate of nabla-g-u-rho-0}
\begin{aligned}
&\frac{1}{2}\frac{d}{dt}\big( \|\p_x^\a g\|^2_{L^2_{x,v}}+ \|\p_x^\a u \|^2_{L^2_{x}} \big)+\frac{C_0}{\v^2}\|\p_x^\a(\I-\P)g\|^2_{\nu}+\frac{1}{\v^2}\|\p_x^\a(b-\v u)\|^2_{L^2_{x}} + \|\p_x^{\a+\a'}u \|^2_{L^2_{x}}\\
&\leq\underbrace{\frac{1}{\v}\l\p_x^{\a} R_0, \p_x^\a g\r_{x,v}}_{B_{11}} + \underbrace{\l\p_x^\a R_1,\p_x^\a g\r_{x,v}}_{B_{12}} + \underbrace{\l\p_x^\a R_2,\p_x^\a u\r_x}_{B_{13}}  \underbrace{- \l\p_x^\a(u_0\c\n_xu+u\c\n_x u_0),\p_x^\a u\r_x}_{B_{14}},
\end{aligned}
\end{equation}
where \eqref{Decop-L} and \eqref{Poin inequality} are utilized.

For $B_{11}$, by recalling the definition of $R_0$ in \eqref{L-R0-R3}, we have
\begin{equation}\label{B11-0}
\begin{aligned}
B_{11}=&\underbrace{-\frac{1}{\v}\l\p_x^\a\big[(\I-\P_0)(v\c\n_x g_1)\big],\p_x^\a g\r_{x,v}}_{B_{111}}\underbrace{-\frac{1}{\v}\l\p_x^\a\big[\frac{1}{\M}\up{div}_v(u_0 g_1 \M)\big],\p_x^\a g\r_{x,v}}_{B_{112}}\\
&\underbrace{-\frac{1}{\v}\l\p_x^\a\big[\frac{1}{\M}\up{div}_v(u_0 g \M)\big],\p_x^\a g\r_{x,v}}_{B_{113}}\underbrace{-\frac{1}{\v}\l\p_x^\a\big[\frac{1}{\M}\up{div}_v(u m_0 M)\big],\p_x^\a g\r_{x,v}}_{B_{114}}.
\end{aligned}
\end{equation}
By noticing $g_1$ in \eqref{L-R0-R3} and $\l\p_x^\alpha(\I-\P_0)f,\p_x^\alpha g\r=\l\p_x^\alpha f,\p_x^\alpha(\I-\P_0)g\r_{x,v}$ in \eqref{self-adjoin of I-P}, we find, for $|\a'|=1$,
\begin{equation*}
\begin{aligned}
B_{111}=&-\frac{1}{\v}\l v \, \p_x^{\a+\a'}g_1,\p_x^\a(\I-\P)g\r_{x,v} - \frac{1}{\v}\l v \, \p_x^{\a+\a'}g_1,v\c\p_x^\a b\M\r_{x,v}\\
=&-\frac{1}{\v}\l v \, \p_x^{\a+\a'}(v\c u_0 m_0)\M,\p_x^\a(\I-\P)g\r_{x,v}-\frac{1}{\v}\l v\, \p_x^{\a+\a'}(v\c u_0)\M,\p_x^\a(\I-\P)g\r_{x,v}\\
&+\frac{1}{\v}\l v \,  \p_x^{\a+\a'}(v\c\n_x m_0)\M,\p_x^\a(\I-\P)g\r_{x,v},
\end{aligned}
\end{equation*}
where $(\I-\P_0)g=(\I-\P_0)(\I-\P)g+v\c b\M$ and $(\I-\P_0)(\I-\P)=\I-\P$ are used.\\
$B_{111}$ is bounded by, for $|\a'|=1$
\begin{equation*}\label{B111}
\begin{aligned}
|B_{111}| \leq &\ \frac{C}{\v}\sum_{ |\tilde{\a}| \leq |\a|}  \binom{\a}{\tilde{\a}} \|(\p_x^{\tilde{\a}}u_0,\p_x^{\tilde{\a}}m_0)\|_{L^4_x} \|(\p_x^{\a+\a'-\tilde{\a}}u_0,\p_x^{\a+\a'-\tilde{\a}}m_0)\|_{L^4_x} \|\p_x^\a(\I-\P)g\|_{L^2_{x,v}} \\
&\ +\frac{C}{\v} \| \p_x^\a u_0\|_{L^2_x} \|\p_x^\a(\I-\P)g\|_{L^2_{x,v}} + \frac{C}{\v} \|\p_x^{\a+\a'}m_0\|_{L^2_x} \|\p_x^\a(\I-\P)g\|_{L^2_{x,v}} \\
\leq&\ \frac{C_0}{2^4\v^2} \|\p_x^\a(\I-\P)g\|^2_{L^2_{x,v}} + C(\D_{ma}(t)+\E^{\frac{1}{2}}(t)\D(t)).
\end{aligned}
\end{equation*}
where Young's inequality is applied in the last inequality.\\
For $B_{112}$, by a direct calculation,
\begin{equation*}\label{B112-0}
\begin{aligned}
B_{112}=&\ \frac{1}{\v}\l\p_x^\a(u_0g_1),\p_x^\a\n_v(\I-\P)g\r_{x,v} + \frac{1}{\v}\l\p_x^\a(u_0 g_1),\p_x^\a\n_v\P g\r_{x,v}\\
&\ +\frac{1}{2\v}\l\p_x^\a(u_0 g_1),v\,\p_x^\a(\I-\P)g\r_{x,v}+\frac{1}{2\v}\l\p_x^\a(u_0 g_1),v\,\p_x^\a\P g\r_{x,v}\\
=&\ \frac{1}{\v}\l\p_x^\a(u_0g_1),\p_x^\a\n_v(\I-\P)g\r_{x,v}+\frac{1}{2\v}\l\p_x^\a(u_0g_1),v\,\p_x^\a(\I-\P)g\r_{x,v},
\end{aligned}
\end{equation*}
and considering the decomposition $g =\P g + (\I-\P) g$, $B_{112}$ is bounded by
\begin{equation*}\label{B112}
\begin{aligned}
|B_{112}| \lesssim &\  \frac{1}{\v} \left(\|u_0\|_{L^6_x} \|u_0\|_{L^6_x}\|m_0\|_{L^6_x}+\|u_0\|_{L^{3}_x}\|u_0\|_{L^6_x}+\|u_0\|_{L^4_x}\|\n_xm_0\|_{L^4_x} \right) \|(\I-\P)g\|_{\nu}\\
\lesssim &\ \frac{1}{\v}(\|\n_xu_0\|^2_{L^2_x} \|\n_xm_0\|_{L^2_x} +\|u_0\|_{H^1_x}\|\n_xu_0\|_{L^2_x} + \|u_0\|_{H^1_x} \|\n_xm_0\|_{H^1_x}) \|(\I-\P)g\|_{\nu}\\[4pt]
\lesssim&\ \E^{\frac{1}{2}}(t)\D(t), \quad \text{for}\ |\a|=0,  \\[8pt]
|B_{112}| \lesssim &\ \frac{1}{\v} \left(\|u_0\|_{H^{|\a|}_x}^2\|\n_x m_0\|_{H^{|\a|-1}_x} + \|u_0\|_{H^{|\a|}_x}\|\n_x u_0\|_{H^{|\a|-1}_x} + \|u_0\|_{H^{|\a|}_x}\|\n_x m_0\|_{H^{|\a|}_x}\right) \|\p_x^\a(\I-\P)g\|_{\nu}\\[4pt]
\lesssim& \ \E^{\frac{1}{2}}(t)\D(t),  \quad \text{for}\ 1 \leq |\a| \leq 4.
\end{aligned}
\end{equation*}
For $B_{113}$,
\begin{equation*}\label{B113}
\begin{aligned}
|B_{113}| \lesssim&\ \frac{1}{\v}\|\n_x u_0\|_{H^1_x}\big[\|(\n_x a,\n_x b)\|_{L^2_x}+\|(\I-\P)g\|_{\nu}\big]\|(\I-\P)g\|_{\nu} - \frac{1}{\v}\l u_0\c\n_v\P g-\frac{v}{2}\c u_0 \P g,\P g\r_{x,v}\\
\lesssim&\ \E^{\frac{1}{2}}(t)\D(t)+\frac{1}{\v}\l u_0 a,b-\v u\r_x+\l u_0 a, u\r_x \\
\lesssim&\ \E^{\frac{1}{2}}(t)\D(t)+\frac{1}{\v}\|u_0\|_{L^3_x}\|a\|_{L^6_x}\|b-\v u\|_{L^2_x} + \|u_0\|_{L^{\frac{3}{2}}_x}\|a\|_{L^6_x}\|u\|_{L^6_x}\\
\lesssim&\ \E^{\frac{1}{2}}(t)\D(t)+\frac{1}{\v}\|u_0\|_{H^1_x}\|\n_x a\|_{L^2_x} \|b-\v u\|_{L^2_x} + \|u_0\|_{H^1_x} \|\n_x a\|_{L^2_x} \|\n_x u\|_{L^2_x} \\[2pt]
\lesssim &\ \E^{\frac{1}{2}}(t)\D(t), \quad \text{for}\ |\a|=0,  \\[8pt]
|B_{113}| \lesssim &\ \frac{1}{\v}\|\n_x u_0\|_{H^{|\a|+1}_x}\big[\|(\n_x a,\n_x b)\|_{H^{|\a|-1}_x} + \sum_{|\a|\leq 3} \|\p_x^\a(\I-\P)g\|_{\nu}\big] \sum_{|\a|\leq 3}\|\p_x^\a(\I-\P)g\|_{\nu}\\
&\ -\frac{1}{\v}\l\p_x^\a(u_0\c\n_v\P g-\frac{v}{2}\c u_0 \P g),\p_x^\a\P g\r_{x,v}\\
\lesssim&\ \E^{\frac{1}{2}}(t)\D(t)+\frac{1}{\v}\l\p_x^\a(u_0 a),\p_x^\a(b-\v u)\r_x+\l\p_x^\a(u_0 a),\p_x^\a u\r_x\\
\lesssim&\ \E^{\frac{1}{2}}(t)\D(t)+\frac{1}{\v}\|\p_x^\a u_0\|_{L^3_x}\|a\|_{L^6_x}\|\p_x^\a(b-\v u)\|_{L^2_x} + \frac{1}{\v}\|u_0\|_{L^\infty_x}\|\p_x^\a a\|_{L^2_x} \|\p_x^\a(b-\v u)\|_{L^2_x} \\
&\ +\frac{1}{\v}\sum_{1 \leq |\tilde{\a}| \leq |\a|}\|\p_x^{\tilde{\a}}u_0\|_{L^4_x} \|\p_x^{\a-\tilde{\a}}a\|_{L^4_x} \|\p_x^\a(b-\v u)\|_{L^2_x} + \|\p_x^\a u_0\|_{L^3_x} \|a\|_{L^6_x} \|\p_x^\a u\|_{L^2_x} \\
&\ +\| u_0\|_{L^\infty_x} \|\p_x^\a a\|_{L^2_x} \|\p_x^\a u\|_{L^2_x} + \sum_{1 \leq |\tilde{\a}| \leq |\a|} \| \p_x^{\tilde{\a}}u_0\|_{L^4_x} \|\p_x^{\a-\tilde{\a}}a\|_{L^4_x} \|\p_x^\a u\|_{L^2_x} \\[2pt]
\lesssim &\ \E^\frac{1}{2}(t)\D(t), \quad \text{for}\ 1 \leq |\a| \leq 4.
\end{aligned}
\end{equation*}
Applying the similar argument as $B_{113}$, we have
\begin{equation*}\label{B114}
|B_{114}| \lesssim \E^\frac{1}{2}(t)\D(t).
\end{equation*}

Therefore, we summarize all the estimates on terms $B_{111}$, $B_{112}$, $B_{113}$ and $B_{114}$ above, and find that there exist constants $C_0, C >0$ such that
\begin{equation}\label{B11}
\begin{aligned}
|B_{11}| \leq \frac{C_0}{2^4\v^2}\|\p_x^\a(\I-\P)g\|^2_{L^2_{x,v}} + C \big(\D_{ma}(t) + \E^{\frac{1}{2}}(t)\D(t)\big).
\end{aligned}
\end{equation}


For $B_{12}$, recalling the definition of $R_1$ and $g_1$ in \eqref{L-R0-R3} and decomposition $g= \P g+(\I-\P)g$, then, for $|\a'|=1$,
\begin{equation}\label{B12-0}
\begin{aligned}
B_{12}=&-\l\p_t\p_x^\a g_1, \p_x^\a g\r_{x,v}-\l\p _x^\a\big[\frac{1}{\M}\up{div}_v(ug_1\M+ug\M)\big],\p_x^\a g\r_{x,v}\\
=&-\l\p_t\p_x^\a g_1,\p_x^\a(\I-\P)g\r_{x,v}-\l\p_t\p_x^\a g_1,\p_x^\a(a\M+v\c b\M)\r_{x,v}\\
&-\l\p_x^\a\big[\frac{1}{\M}\up{div}_v(ug_1\M+ug\M)\big],\p_x^\a g\r_{x,v}\\
=&-\l\p_t\p_x^\a g_1,v\c\p_x^\a b\M\r_{x,v}-\l\p _x^\a\big[\frac{1}{\M}\up{div}_v(ug_1\M+ug\M)\big],\p_x^\a g\r_{x,v}\\
=&\underbrace{-\l\p_t \p_x^\a(m_0u_0),\p^\a_x b\r_x}_{B_{121}} \underbrace{-\l \p_t\p_x^\a u_0+\p_t\p_x^{\a+\a'} m_0,\p^\a_x b\r_x}_{B_{122}} \\
& \underbrace{-\l\p_x^\a\big[\frac{1}{\M}\up{div}_v(ug_1\M+ug\M)\big],\p_x^\a g\r_{x,v}}_{B_{123}}.
\end{aligned}
\end{equation}
For $B_{121}$,
\begin{equation*}\label{B121}
\begin{aligned}
|B_{121}|=&\  |- \l (\p_t m_0) u_0 + m_0 (\p_t u_0), b-\v u\r_x - \v\l (\p_t m_0) u_0+m_0(\p_t u_0), u\r_x| \\[4pt]
\leq&\ \big( \|\p_tm_0\|_{L^2_x} \|u_0\|_{L^\infty_x} + \|\p_t u_0\|_{L^2_x} \| m_0\|_{L^\infty_x}\big) \|b-\v u\|_{L^2_x} \\
& \qquad \qquad \qquad \qquad \qquad \qquad  + \v\big( \|\p_tm_0\|_{L^2_x} \|u_0\|_{L^3_x}+\|\p_t u_0\|_{L^2_x} \| m_0\|_{L^3_x}\big) \|u\|_{L^6_x}\\[2pt]
\lesssim &\  \big(\|\p_tm_0\|_{L^2_x} \|u_0\|_{H^2_x}+\|\p_t u_0\|_{L^2_x} \| m_0\|_{H^2_x}\big) \|b-\v u\|_{L^2_x}\\
& \qquad \qquad \qquad \qquad \qquad \qquad  +  \v\big(\|\p_tm_0\|_{L^2_x} \|u_0\|_{H^1_x}+\|\p_t u_0\|_{L^2_x} \| m_0\|_{H^1_x} \big)\|\n_x u\|_{L^2_x} \\
\lesssim&\ \E^{\frac{1}{2}}(t)\D(t), \quad \text{for} \ |\a|=0, \\[8pt]
|B_{121}|=&|-\l\p_t \p_x^\a(m_0u_0),\p^\a_x(b-\v u)\r_x - \v\l\p_t \p_x^\a(m_0u_0),\p^\a_x u\r_x|\\[4pt]
\lesssim& \ \big(\|\p_tu_0\|_{H^{|\a|}_x}\|m_0\|_{H^{|\a|}_x}+\|u_0\|_{H^{|\a|}_x}\|\p_t m_0\|_{H^{|\a|}_x}\big) \|\p_x^\a(b-\v u)\|_{L^2_x} \\[2pt]
&\ +\v\big(\|\p_t u_0\|_{H^{|\a|}_x} \| m_0\|_{H^{|\a|}_x} \|\n_x u\|_{H^{|\a|-1}_x} + \|\p_t m_0\|_{H^{|\a|}_x} \| u_0\|_{H^{|\a|}_x} \|\n_x u\|_{H^{|\a|-1}_x}\big)\\[2pt]
\lesssim&\ \E^{\frac{1}{2}}(t)\D(t), \quad \text{for}\ 1 \leq |\a| \leq 4,
\end{aligned}
\end{equation*}
For $B_{122}$, we have, for $|\a'|=1$,
\begin{equation*}\label{B122}
\begin{aligned}
|B_{122}|=&\ |-\l\p_t u_0+\p_t\n_x m_0,b-\v u\r_x-\v\l\p_t u_0+\p_t\n_x m_0, u\r_x | \\
\leq&\ \frac{1}{2^4\v^2} \|b-\v u\|^2_{L^2_x} + C\|\p_t u_0+\p_t\n_x m_0\|^2_{L^2_x} + \v\|\p_t u_0+\p_t\n_x m_0\|_{L^{\frac{6}{5}}_x} \|u\|_{L^6_x}\\
\leq&\ \frac{1}{2^4\v^2} \|b-\v u\|^2_{L^2_x} + C\|(\p_t u_0,\p_t\n_x m_0)\|^2_{L^2_x} + C\|(\p_t u_0,\p_t\n_x m_0)\|^2_{H^1_x}+\frac{1}{2^4}\|\n_xu\|^2_{L^2_x} \\
\leq&\ C\D_{ma}(t)+\frac{1}{2^4\v^2} \|b-\v u\|^2_{L^2_x} + \frac{1}{2^4}\|\n_xu\|^2_{L^2_x}, \quad \text{for} \ |\a|=0,  \\[8pt]
|B_{122}|=& \ |-\l\p_t \n_x^\a u_0+\p_t\n_x^{\a+\a'}m_0,\p_x^\a(b-\v u)\r_x-\v\l\p_t\p_x^\a u_0+\p_t\p_x^{\a+\a'}m_0, \n_x^\a u\r_x |\\
\leq&\ \frac{1}{4^3\v^2} \|\p_x^\a(b-\v u)\|^2_{L^2_x} + C\|\p_t\p_x^\a u_0\|^2_{L^2_x} + \v\|\p_t\n_x u_0\|_{H^{|\a|-1}_x} \| \n_x u\|_{H^{|\a|-1}_x}\\
\leq &\ C \D_{ma}(t)+\frac{1}{2^4\v^2} \|\p_x^\a(b-\v u)\|^2_{L^2_x} + \frac{1}{2^4}\|\p_x^{\a+\a'}u\|^2_{L^2_x}, \quad \text{for}\ 1 \leq |\a| \leq 4.
\end{aligned}
\end{equation*}
Combining the estimates for $B_{121}$ and $B_{122}$ above, we have, for $|\a'|=1$,
\begin{multline}\label{B121-B122}
    \Big| -\l\p_t\p_x^\a g_1,v\c\p_x^\a b\M\r_{x,v}\Big| \leq \ C \left(\E^{\frac{1}{2}}(t)\D(t) + \D_{ma}(t) \right)
    + \frac{1}{2^4\v^2}\|\p_x^\a(b-\v u)\|^2_{L^2_x} + \frac{1}{2^4}\|\p_x^{\a+\a'}u\|^2_{L^2_x}.
\end{multline}
Applying the similar argument as $B_{113}$ in \eqref{B113}, $B_{123}$ is bounded by
\begin{equation*}\label{B123}
\begin{aligned}
|B_{123}|=& \ \Big| -\l\p_x^\a\big[\frac{1}{\M}\up{div}_v(ug_1\M+ug\M)\big],\p_x^\a g\r_{x,v} \Big|\\
\lesssim& \ \E^{\frac{1}{2}}(t)\D(t).
\end{aligned}
\end{equation*}

Hence, collecting the previous estimations on $B_{121}$, $B_{123}$ and $B_{123}$, we obtain, for $|\a'|=1$,
\begin{equation}\label{B12}
\begin{aligned}
|B_{12}| \leq C \left( \D_{ma}(t)+\E^{\frac{1}{2}}(t)\D(t) \right) + \frac{1}{2^4\v^2} \|\p_x^\a(b-\v u)\|^2_{L^2_x} + \frac{1}{2^4}\|\p_x^{\a+\a'}u\|^2_{L^2_x}.
\end{aligned}
\end{equation}


By recalling the definition of $R_2$ in \eqref{L-R0-R3} and the similar estimates of $B_{11}$, for $B_{13}$ and $B_{14}$,
\begin{equation}\label{B13-B14}
\begin{aligned}
|B_{13}| + |B_{14}| \lesssim &\ \E^{\frac{1}{2}}(t)\D(t).
\end{aligned}
\end{equation}

Finally, inserting the estimates \eqref{B11}, \eqref{B12}, \eqref{B13-B14} into \eqref{Estimate of nabla-g-u-rho-0}, and then summing up $\a$ from 0 to 4, the proof of the lemma is completed.
\end{proof}

Next, we deduce \textit{a priori} estimates about mixed partial derivative $(\I-\P)g$. Taking the projection operator $\I-\P$ to the remainder equation $\eqref{Remainder equations}_1$, it follows that
\begin{equation}\label{Remainder-equation-I-P-g}
\begin{aligned}
\p_t(\I-\P)g+\frac{1}{\v^2}L(\I-\P)g=\frac{1}{\v}(\I-\P)R_0+(\I-\P)R_1,
\end{aligned}
\end{equation}
where the fact $L g=L(\I-\P)g+\P_1 g$ in \eqref{Decop-L} is used, and $R_0,\,R_1$ are given as in \eqref{L-R0-R3}.


The lemma of \textit{a priori} estimates about mixed partial derivative $(\I-\P)g$ is shown as follows:

\begin{lemma}\label{Estiamte-Mi-kinetic-part-mixed-partial}
Let $(g,u)$ be the solution to the remainder system \eqref{Remainder equations}, then there exist constants $C_2,\,C_{1,2} > 0$ independent of $\v$ such that, for $t \geq 0$,
\begin{equation}\label{Step two}
\frac{1}{2}\frac{d}{dt}\E_{mi,K,2}(t)+C_{1,2}\D_{mi,K,2}(t)\leq C_2 \big( \D_{mi,K,1}(t)+\D_{mi,F}(t)+\D_{ma}(t)+\E^{\frac{1}{2}}(t)\D(t)\big),
\end{equation}
where the energy and dissipation functionals $\E_{mi,K,2}(t)$, $\D_{mi,K,2}(t)$, $\D_{mi,K,1}(t)$, $\D_{mi,F}(t)$, $\D_{ma}(t)$, $\E(t)$, and $\D(t)$ are defined in \eqref{macro energy functionals and energy dissipation functionals}, \eqref{part of energy functionals}, and \eqref{energy functionals and energy dissipation functional}, respectively.
\end{lemma}

\begin{proof}
By applying $\p_x^\a \p_v^\b$ with $1\leq|\a|+|\b|\leq 4$ and $ |\b| \geq 1 $ to the system \eqref{Remainder-equation-I-P-g}, multiplying by $\p_x^\a\p_v^\b(\I-\P) g$, and integrating over $\mathbb{R}^3 \times \mathbb{R}^3 $, we obtain, for $|\b'|=1$ and $|\b''|=2$,
\begin{equation}\label{Estimate of mix-partial-I-P-g-1}
\begin{aligned}
&\frac{1}{2}\frac{d}{dt}\|\p_x^\a\p_v^\b(\I-\P)g\|^2_{L^2_{x,v}} + \frac{C_0}{\v^2}\|\p_x^\a\p_v^\b(\I-\P)g\|^2_{\nu} \\
\leq &\ \underbrace{-\frac{ \binom{\b}{\b'} }{2\v^2}\l v\p_x^\a\p_v^{\b-\b'}(\I-\P)g,\p_x^\a\p_v^\b(\I-\P)g\r_{x,v}}_{B_{21}} \underbrace{-\frac{\binom{\b}{\b''}}{2\v^2}\l \p_x^\a\p_v^{\b-\b''}(\I-\P)g,\p_x^\a\p_v^{\b-\b''}\Delta_v(\I-\P)g\r_{x,v}}_{B_{22}} \\
&\ +\underbrace{\frac{1}{\v}\l\p_x^\a\p_v^\b(\I-\P)R_0,\p_x^\a\p_v^\b(\I-\P)g\r_{x,v}}_{B_{23}} + \underbrace{\l\p_x^\a\p_v^\b(\I-\P)R_1,\p_x^\a\p_v^\b(\I-\P)g\r_{x,v}}_{B_{24}}
\end{aligned}
\end{equation}
where the inequality \eqref{Poin inequality} and the following calculation are used
\begin{equation*}
    \p_x^\a\p_v^\b L(\I-\P)g= L[\p_x^\a\p_v^\b(\I-\P)g] + \frac{\binom{\b}{\b'}}{2} v\p_x^\a\p_v^{\b-\b'}(\I-\P)g + \frac{\binom{\b}{\b''}}{2}\p_x^\a\p_v^{\b-\b''}(\I-\P)g I.
\end{equation*}

For $B_{21}$ and $B_{22}$, by applying integration by parts, for $|\b'|=1$ and $|\b''|=2$,
\begin{equation}\label{B41-B42}
\begin{aligned}
|B_{21}| + |B_{22}|= &\ \Big|-\frac{\binom{\b}{\b'}}{2\v^2}\l v\p_x^\a\p_v^{\b-\b'}(\I-\P)g,\p_x^\a\p_v^\b(\I-\P)g\r_{x,v} \Big| \\
& \qquad \qquad \qquad \qquad \qquad + \Big| -\frac{\binom{\b}{\b''}}{2\v^2}\l\p_x^\a\p_v^{\b-\b''}(\I-\P)g,\p_x^\a\p_v^{\b-\b''}\Delta_v(\I-\P)g\r_{x,v} \Big|\\
\leq&\ \frac{C_0}{2^6\v^2}\|\p_x^\a\p_v^\b(\I-\P)g\|_{\nu}^2+\frac{C}{\v^2}\|\p_x^\a\p_v^{\b-\b'}(\I-\P)g\|^2_{L^2_{x,v}},
\end{aligned}
\end{equation}
where the Young inequality and H$\ddot{\up{o}}$lder inequality are used in the inequality.

Recalling $R_0$ in \eqref{L-R0-R3}, and $(\I-\P_0)(\I-\P)g=(\I-\P)g$ in \eqref{I-P}, $B_{23}$ can be divided into five parts:
\begin{equation}\label{B43-0}
\begin{aligned}
B_{23}=&\underbrace{-\frac{1}{\v}\l\p_x^\a\p_v^\b(\I-\P)(v\c\n_x g_1),\p_x^\a\p_v^\b(\I-\P)g\r_{x,v}}_{B_{231}}\underbrace{-\frac{1}{\v}\l\p_x^\a\p_v^\b(\I-\P)(v\c\n_x g),\p_x^\a\p_v^\b(\I-\P)g\r_{x,v}}_{B_{232}}\\
&\underbrace{-\frac{1}{\v}\l\p_x^\a\p_v^\b(\I-\P)\big[\frac{1}{\M}\up{div}_v(u_0 g_1 \M)\big],\p_x^\a\p_v^\b(\I-\P)g\r_{x,v}}_{B_{233}}\\
&\underbrace{-\frac{1}{\v}\l\p_x^\a\p_v^\b(\I-\P)\big[\frac{1}{\M}\up{div}_v(u_0 g \M)\big],\p_x^\a\p_v^\b(\I-\P)g\r_{x,v}}_{B_{234}}\\
&\underbrace{- \frac{1}{\v}\l\p_x^\a\p_v^\b(\I-\P)\big[\frac{1}{\M}\up{div}_v(u m_0 M)\big],\p_x^\a\p_v^\b(\I-\P)g\r_{x,v}}_{B_{235}}.
\end{aligned}
\end{equation}
For $B_{231}$, by considering the Sobolev embedding and H$\ddot{\up{o}}$lder inequality, we have
\begin{equation*}\label{B431}
\begin{aligned}
|B_{231}|\leq&\ \frac{1}{\v} \|\p_x^\a\p_v^\b(v\c\n_x g_1)\|_{L^2_{x,v}} \|\p_x^\a\p_v^\b(\I-\P)g\|_{L^2_{x,v}}  \\
\leq&\ \frac{C_0}{2^6\v^2}\|\p_x^\a\p_v^\b(\I-\P)g\|^2_{L^2_{x,v}} + C\big(\D_{ma}(t)+\E^{\frac{1}{2}}(t)\D(t)\big).
\end{aligned}
\end{equation*}
For $B_{232}$, since $(\I-\P)(v\c\n_x g)=v\c\n_x g-\P(v\c\n_x g)=v\c\n_x(\I-\P)g+v\c\n_x\P g-\P(v\c\n_x g)$,
\begin{equation*}
\begin{aligned}
B_{232}=&\ -\frac{1}{\v}\l\p_x^\a\p_v^\b(\I-\P)(v\c\n_x g),\p_x^\a\p_v^\b(\I-\P)g\r_{x,v}\\
=&\ \underbrace{-\frac{1}{\v}\l\p_x^\a\p_v^\b v\c\n_x (\I-\P)g,\p_x^\a\p_v^\b(\I-\P)g\r_{x,v}}_{B_{2321}} \underbrace{-\frac{1}{\v}\l\p_x^\a\p_v^\b v\c\n_x \P g,\p_x^\a\p_v^\b(\I-\P)g\r_{x,v}}_{B_{2322}} \\
& +\underbrace{\frac{1}{\v}\l\p_x^\a\p_v^\b \P (v\c\n_x g),\p_x^\a\p_v^\b(\I-\P)g\r_{x,v}}_{B_{2323}}.
\end{aligned}
\end{equation*}
For $B_{2321}$ and $B_{2322}$, we have, for $|\a'|=|\b'|=1$ and $|\b''|=2$,
\begin{equation*}
\begin{aligned}
|B_{2321}|=&\ \Big|-\frac{\binom{\b}{\b'}}{\v} \l\p_x^{\a+\a'}\p_v^{\b-\b'}(\I-\P)g,\p_x^\a\p_v^\b(\I-\P)g\r_{x,v} \Big|\\
\leq&\ \frac{C_0}{2^8\v^2} \|\p_x^\a\p_v^\b(\I-\P)g\|^2_{L^2_{x,v}} + C\|\p_x^{\a+\a'}\p_v^{\b-\b'}(\I-\P)g\|^2_{L^2_{x,v}}.
\end{aligned}
\end{equation*}
\begin{equation*}
\begin{aligned}
|B_{2322}|=&\ \Big| -\frac{1}{\v}\l\p_x^\a\p_v^\b[(v\c\n_x a+v\otimes v:\n_x b)\M],\p_x^\a\p_v^\b(\I-\P)g\r_{x,v} \Big|\\
\leq&\ \frac{C_0}{2^8\v^2} \|\p_x^\a\p_v^\b(\I-\P)g\|^2_{L^2_{x,v}} + C\|(\p_x^{\a+\a'}a,\p_x^{\a+\a'}b)\|^2_{L^2_{x}}\\
\leq&\ \frac{C_0}{2^8\v^2} \|\p_x^\a\p_v^\b(\I-\P)g\|^2_{L^2_{x,v}} + C\D_{mi,F}(t).
\end{aligned}
\end{equation*}
Furthermore, for $B_{2323}$, noticing that
\begin{equation*}
\begin{aligned}
\P(v\c\n_x g)=&\ \P[v\c\n_x\P g]+\P[v\c\n_x(\I-\P)g]\\
=& \ (v\c\n_x a+\up{div}_x b)\M+\P[v\c\n_x(\I-\P)g]\\
=&\ \big[v\c\n_x a+\up{div}_x b+\l v\c\n_x(\I-\P)g, \M\r_v+v\c\l v\c\n_x(\I-\P)g,v\M\r_v\big]\M,
\end{aligned}
\end{equation*}
we have, for $|\a'|=1$,
\begin{equation*}
\begin{aligned}
|B_{2323}|=&\ \Big|\frac{1}{\v} \l \p_x^\a\p_v^\b \P (v\c\n_x g), \p_x^\a\p_v^\b(\I-\P)g \r_{x,v} \Big|\\
\leq &\ \frac{C_0}{2^8\v^2} \|\p_x^\a\p_v^\b(\I-\P)g\|^2_{L^2_{x,v}}  +C\|(\p_x^{\a+\a'}a,\p_x^\a\up{div}_x b)\|^2_{L^2_{x}} +C\|\p_x^{\a+\a'}(\I-\P)g\|^2_{L^2_{x,v}} \\
\leq &\ \frac{C_0}{2^8\v^2}\|\p_x^\a\p_v^\b(\I-\P)g\|^2_{L^2_{x,v}} + C\big(\D_{mi,F}(t)+\D_{mi,K,1}(t)\big).
\end{aligned}
\end{equation*}
Therefore, we have, for $|\a'|=|\b'|=1$,
\begin{equation*}\label{B432}
\begin{aligned}
|B_{232}|\leq &\ \frac{C_0}{2^6\v^2} \|\p_x^\a\p_v^\b(\I-\P)g\|^2_{L^2_{x,v}} + C\big[\|\p_x^{\a+\a'}\p_v^{\b-\b'}(\I-\P)g\|^2_{L^2_{x,v}} + \D_{mi,F}(t)+\D_{mi,K,1}(t)\big],
\end{aligned}
\end{equation*}
where $\D_{mi,K,1}(t)$, $\D_{mi,F}(t)$ are defined in \eqref{part of energy functionals}.\\
For $B_{233}$ to $B_{235}$, by a direct calculation,
\begin{equation*}
    \P\big[\frac{1}{\M}\up{div}_v(u_0 g_1 \M)\big]=0 \quad \text{and} \quad \P\big[\frac{1}{\M}\up{div}_v(u_0 g \M)\big]=- (v \c u_0) a\M,
\end{equation*}
we have
\begin{equation*}\label{B433}
\begin{aligned}
|B_{233}|=&\ \Big| -\frac{1}{\v}\l\p_x^\a\p_v^\b\big[\frac{1}{\M}\up{div}_v(u_0 g_1 \M)\big],\p_x^\a\p_v^\b(\I-\P)g\r_{x,v} \Big|\\
=&\ \Big| -\frac{1}{\v}\l \p_x^\a\p_v^\b [u_0\c\n_v g_1-\frac{1}{2} (v\c u_0) g_1],\p_x^\a\p_v^\b(\I-\P)g \r_{x,v} \Big|\\
\leq &\ \frac{1}{\v}\| \p_x^\a\p_v^\b [u_0\c\n_v g_1-\frac{1}{2} (v\c u_0) g_1]\|_{L^2_{x,v}}  \|\p_x^\a\p_v^\b(\I-\P)g\|_{L^2_{x,v}} \\
\lesssim &\ \E^{\frac{1}{2}}(t)\D(t).
\end{aligned}
\end{equation*}
\begin{equation*}\label{B434}
\begin{aligned}
|B_{234}|=&\ \Big| -\frac{1}{\v}\l\p_x^\a\p_v^\b\big\{\frac{1}{\M}\up{div}_v[u_0(\I-\P)g \M]\big\},\p_x^\a\p_v^\b(\I-\P)g\r_{x,v} \Big|\\
&\ + \Big|-\frac{1}{\v}\l\p_x^\a\p_v^\b\big\{\frac{1}{\M}\up{div}_v[u_0 (a+v\c b)M]\big\},\p_x^\a\p_v^\b(\I-\P)g\r_{x,v}\Big|\\
&\ + \Big|\frac{1}{\v}\l \p_x^\a\p_v^\b[(v\c u_0) a\M], \p_x^\a\p_v^\b(\I-\P)g\r_{x,v}\Big|\\
\leq&\ \frac{1}{\v}\|\p_x^\a\p_v^\b\big\{\frac{1}{\M}\up{div}_v[u_0(\I-\P)g \M]\big\}\|_{L^2_{x,v}} \|\p_x^\a\p_v^\b(\I-\P)g\|_{L^2_{x,v}} \\
&\ +\frac{1}{\v}\|\p_x^\a\p_v^\b\big\{\frac{1}{\M}\up{div}_v[u_0 (a+v\c b)M]\big\}\|_{L^2_{x,v}} \|\p_x^\a\p_v^\b(\I-\P)g\|_{L^2_{x,v}}\\
&\ +\|\p_x^\a\p_v^\b\big\{\frac{1}{\M}\up{div}_v[u_0 (a+v\c b)M]\big\}\|_{L^2_{x,v}} \|\p_x^\a\p_v^\b(\I-\P)g\|_{L^2_{x,v}} \\
\lesssim &\ \E^\frac{1}{2}(t)\D(t).
\end{aligned}
\end{equation*}
\begin{equation*}\label{B435}
\begin{aligned}
|B_{235}|=&\ \Big|-\frac{1}{\v}\l\p_x^\a\p_v^\b\big[\frac{1}{\M}\up{div}_v(u m_0 M)\big],\p_x^\a\p_v^\b(\I-\P)g\r_{x,v}|\\
&\ +|\frac{1}{\v}\l\p_x^\a\p_v^\b[(v\c u) m_0\M],\p_x^\a\p_v^\b(\I-\P)g\r_{x,v}\Big|\\
\leq&\ \frac{1}{\v}\|\p_x^\a\p_v^\b\big[\frac{1}{\M}\up{div}_v(u m_0 M)\big]\|_{L^2_{x,v}} \|\p_x^\a\p_v^\b(\I-\P)g\|_{L^2_{x,v}} \\
&\ +\frac{1}{\v}\|\p_x^\a\p_v^\b[(v\c u) m_0\M]\|_{L^2_{x,v}} \|\p_x^\a\p_v^\b(\I-\P)g\|_{L^2_{x,v}}\\
\lesssim &\ \E^\frac{1}{2}(t)\D(t).
\end{aligned}
\end{equation*}

By collecting the estimates on $B_{231}$, $B_{232}$, $B_{233}$, $B_{234}$, $B_{235}$ above and inserting them into \eqref{B43-0}, we obtain the estimate for $B_{23}$ that, for $|\a'|=|\b'|=1$,
\begin{multline}\label{B43}
|B_{23}| \leq C\big[\|\p_x^{\a+\a'}\p_v^{\b-\b'}(\I-\P)g\|^2_{L^2_{x,v}} + \D_{mi,F}(t)+\D_{mi,K,1}(t)+\E^\frac{1}{2}(t)\D(t)\big] +\frac{C_0}{2^5\v^2}\|\p_x^\a\p_v^\b(\I-\P)g\|^2_{L^2_{x,v}}.
\end{multline}


Recalling the definition of $R_1$ in \eqref{L-R0-R3}, using a similar estimate as $B_{23}$, for $B_{24}$, we have
\begin{equation}\label{B44}
\begin{aligned}
|B_{24}|\leq &\ C\big(\D_{mi,F}(t)+\E^\frac{1}{2}(t)\D(t)\big) + \frac{C_0}{2^6\v^2}\|\p_x^\a\p_v^\b(\I-\P)g\|^2_{L^2_{x,v}}.
\end{aligned}
\end{equation}

Notice that the quantity $\frac{C}{\v^2}\|\p_x^\a\p_v^{\b-\b'}(\I-\P)g\|^2_{L^2_{x,v}}$ in \eqref{B41-B42} is still not controlled. However, by observing that the orders of $v$-derivatives in this quantity is $|\b|-1$, we can employ an induction over $|\b|$ and combine \eqref{Estimate of mix-partial-I-P-g-1}, \eqref{B41-B42}, \eqref{B43}, \eqref{B44} to find $C_{1,2}\, C_2 >0 $ such that
\begin{equation}\label{Estimate of mix-partial-I-P-g}
\begin{aligned}
\frac{1}{2}\frac{d}{dt}\E_{mi,K,2}(t) + C_{1,2}\D_{mi,K,2}(t)\leq C_2\big( \D_{mi,K,1}(t)+\D_{mi,F}(t)+\D_{ma}(t) + \E^{\frac{1}{2}}(t)\D(t)\big).
\end{aligned}
\end{equation}

This completes the proof of the lemma.

\end{proof}


\subsection{\textit{A priori} estimates for the macroscopic part of the remainder system \eqref{Remainder equations}}
\label{subsec:remainder_fluid}

The goal of this subsection is to obtain the energy dissipation rate of $\frac{1}{\v} \| \n_x b\|^2_{L^2_x}+\|\n_x a\|^2_{L^2_x}$. Inspired by \cite{CJADRJMA11}, we propose an auxiliary hyperbolic-parabolic coupled system of $a$ and $b$:
\begin{equation}\label{Pg-a-b}
\left\{
\begin{aligned}
&\v\p_t a+\up{div}_x b=0,\\
&\p_t b_i+\frac{1}{\v}\p_{x_i} a + \frac{1}{\v^2} (b_i-\v u_i)+\frac{1}{\v}\sum_{j=1}^3\p_{x_j}\Gamma_{ij}[(\I-\P)g]= \frac{1}{\v}R_{4i}+R_{5i},\\
&\frac{1}{\v}(\p_{x_i}b_j+\p_{x_j}b_i)-\big[(\frac{1}{\v}u_{0i}+u_i)b_j+(\frac{1}{\v}u_{0j}+u_j)b_i\big] =- \p_t \Gamma_{ij}[(\I-\P)g]-\frac{2}{\v^2}\Gamma_{ij}[(\I-\P)g]+\Gamma_{ij}(l),\\
\end{aligned}
\right.
\end{equation}
for $1\leq i,j\leq 3$, where $a, b$ are defined as in \eqref{ab}, $R_{4i}$, $R_{5i}$, $l$ and $\Gamma_{ij}(g)$ are given by
\begin{equation}\label{R4i-R5i-Gamma-l}
\begin{aligned}
R_{4i}=&\ u_{0i}a+u_i m_0,\\
R_{5i}=&\ -\big[\p_t(u_{0i} m_0)+u_i a+\p_t u_{0i}+\p_t\p_{x_i} m_0\big],\\
l=&\ -\frac{1}{\v}\big[(\I-\P_0) (v\c\n_x g_1) + v\c\n_x(\I-\P)g+\frac{1}{\M}\up{div}_v(u_0 g_1 \M)\big]-\frac{1}{\M}\up{div}_v(u g_1\M),\\
\Gamma_{ij}(g)=&\ \int g(v_i v_j-1)\M dv.
\end{aligned}
\end{equation}

\begin{lemma}\label{Estiamte-Mi-fluid-part}
Let $(g,u)$ be the solution to the remainder system \eqref{Remainder equations}, there exist constants $C_3,\,C_{1,3} > 0$ independent of $\v$ and $t$ such that, for any $t \geq 0$,
\begin{equation}\label{Last step}
\begin{aligned}
\frac{1}{2}\frac{d}{dt}\E_{mi,F}(t)+C_{1,3}\D_{mi,F}(t)\leq C_3 \big( \D_{mi,K,1}(t)+\D_{ma}(t)+\E^{\frac{1}{2}}(t)\D(t)\big),
\end{aligned}
\end{equation}
where the energy and dissipation functionals $\E_{mi,F}(t)$, $\D_{mi,F}(t)$, $\D_{mi,K,1}(t)$, $\D_{ma}(t)$, $\E(t)$ and $\D(t)$ are defined in \eqref{macro energy functionals and energy dissipation functionals}, \eqref{part of energy functionals}, and \eqref{energy functionals and energy dissipation functional}, respectively.
\end{lemma}

\begin{proof}
By applying the derivative operator $\p_x^\a$ with $0\leq|\a|\leq 3$ to the hyperbolic-parabolic system $\eqref{Pg-a-b}_2$, taking $L^2_{x}$-inner product with $\p_x^\a b$, and then integrating by parts, we have, for $|\a'|=1$,
\begin{multline}\label{Energy of b}
\frac{1}{2}\|\n_x^\a b\|^2_{L^2_x} + \underbrace{\frac{1}{\v}\l\p_x^{\a+\a'} a,\p_x^\a b\r_x}_{B_{31}} + \underbrace{\frac{1}{\v^2}\l\p_x^{\a} (b-\v u),\p_x^\a b\r_x}_{B_{32}} + \underbrace{\frac{1}{\v}\sum_{i,j=1}^3\l\p_x^{\a}\Gamma_{ij}[\p_{x_j}(\I-\P)g],\p_x^\a b\r_x}_{B_{33}} \\
=\underbrace{\frac{1}{\v}\l\p_x^\a R_4,\p_x^\a b\r_x}_{B_{34}} + \underbrace{\l\p_x^\a R_5,\p_x^\a b\r_x}_{B_{35}}.
\end{multline}

For $B_{31}$, by noticing $\p_t a=-\frac{1}{\v}\up{div}_xb$ in $\eqref{Pg-a-b}_1$, we have, for $|\a'|=1$,
\begin{equation*}\label{Energy of a}
\begin{aligned}
B_{31}= \frac{1}{\v}\l\p_x^{\a+\a'}a,\p_x^\a b\r_x = -\frac{1}{\v}\l\p_x^\a a,\p_x^\a\up{div}_xb\r_x=\l\p_x^\a a,\p_x^\a\p_t a\r_x=\frac{1}{2}\frac{d}{dt}\|\p_x^\a a\|^2_{L^2_x}.
\end{aligned}
\end{equation*}

For $B_{32}$,
\begin{equation*}\label{B32}
\begin{aligned}
|B_{32}| =&\ \Big| \frac{1}{\v^2}\l b-\v u,b-\v u\r_x + \frac{1}{\v}\l b-\v u, u\r_x \Big| \\
\leq& \ \frac{1}{\v^2}\| b-\v u\|^2_{L^2_x} + \frac{1}{\v}\| b-\v u\|_{L^{\frac{6}{5}}_x}\|u\|_{L^6_x}\\
\lesssim &\ \frac{1}{\v^2} \| b-\v u\|_{H^1_x}^2 + \|\n_x u\|^2_{L^2_x}\\
\lesssim&\ \D_{mi,K,1}(t), \quad \text{for} \ |\a|=0, \\[8pt]
|B_{32}| =&\ \Big| \frac{1}{\v^2}\l\p_x^\a(b-\v u),\p_x^\a(b-\v u)\r_x+\frac{1}{\v}\l\p_x^\a(b-\v u),\p_x^\a u\r_x \Big| \\
\leq& \ \frac{1}{\v^2}\|\p_x^\a(b-\v u)\|^2_{L^2_x} +\frac{1}{\v}\|\p_x^\a(b-\v u)\|_{L^2_x} \|\p_x^\a u\|_{L^2_x} \\
\lesssim&\ \frac{1}{\v^2}\|\p_x^\a(b-\v u)\|_{H^1_x}^2 + \|\p_x^\a u\|^2_{L^2_x}\\
\lesssim&\ \D_{mi,K,1}(t), \quad \text{for} \ 1 \leq |\a| \leq 3. \\
\end{aligned}
\end{equation*}

For $B_{33}$, we have, for $|\a'|=1$,
\begin{equation*}\label{B33}
\begin{aligned}
|B_{33}| =&\ \Big| \frac{1}{\v}\sum_{i,j=1}^3\l\Gamma_{ij}[\p_{x_j}(\I-\P)g],(b-\v u)\r_x+\sum_{i,j=1}^3\l\Gamma_{ij}[\p_{x_j}(\I-\P)g], u\r_x \Big|\\
\lesssim &\ \frac{1}{\v}\big[\|\n_x(\I-\P)g\|^2_{L^2_{x,v}} + \|b-\v u\|^2_{L^2_x}\big] + \|\n_x(\I-\P)g\|_{L^{\frac{6}{5}}_x}\| u\|_{L^6_x}\\[2pt]
\lesssim &\ \D_{mi,K,1}(t), \quad \text{for} \ |\a|=0,\\[8pt]
|B_{33}| =&\ \Big|\frac{1}{\v}\sum_{i,j=1}^3\l\p_x^{\a}\Gamma_{ij}[\p_{x_j}(\I-\P)g],\p_x^\a (b-\v u)\r_x + \sum_{i,j=1}^3\l\p_x^{\a}\Gamma_{ij}[\p_{x_j}(\I-\P)g],\p_x^\a u\r_x \Big|\\
\lesssim&\ \frac{1}{\v}\big[\|\p_x^{\a+\a'}(\I-\P)g\|^2_{L^2_{x,v}} + \|\p_x^\a(b-\v u)\|^2_{L^2_x} \big] + \|\p_x^{\a+\a'}(\I-\P)g\|_{L^2_{x,v}} \|\p_x^{\a}u\|_{L^2_x}\\[2pt]
\lesssim &\ \D_{mi,K,1}(t), \quad \text{for} \ 1 \leq |\a| \leq 3.
\end{aligned}
\end{equation*}

For $B_{34}$ and $B_{35}$, by following the same argument as $B_{31}$, we have, for $|\a'|=1$,
\begin{equation*}\label{B34-B35}
\begin{aligned}
|B_{34}| + |B_{35}| \leq \frac{1}{2\v}\|\p_x^{\a+\a'} b\|_{L^2_x} + C(\D_{ma}(t)+\D_{mi,K,1}(t)+\E^{\frac{1}{2}}\D(t)).
\end{aligned}
\end{equation*}

Hence, substituting the estimates of $B_{31}$, $B_{32}$, $B_{33}$, $B_{34}$ and $B_{35}$ above into \eqref{Energy of b}, we find, for $|\a'|=1$,
\begin{equation}\label{Energy of a and b}
\begin{aligned}
\frac{1}{2}\frac{d}{dt}\|(\p_x^\a a,\p_x^\a b)\|^2_{L^2_x} \leq \frac{1}{2\v}\|\p_x^{\a+\a'}b\|^2_{L^2_x} + C\big(\D_{ma}(t)+\D_{mi,K,1}(t)+\E^{\frac{1}{2}}\D(t)\big).
\end{aligned}
\end{equation}

Furthermore, it follows from $\eqref{Pg-a-b}_3$ that
\begin{multline}\label{partial-b-1}
\frac{1}{\v}\sum_{i,j=1}^3\|\p_x^\a(\p_{x_i} b+\p_{x_j}b_i)\|^2_{L^2_x} = -\frac{d}{dt}\sum_{i,j=1}^3\int\p_x^\a(\p_{x_i} b+\p_{x_j}b_i)\p_x^\a\Gamma_{ij}[(\I-\P)g]dx\\
 +\underbrace{\sum_{i,j=1}^3\l\p_x^\a (\p_{x_j}\p_tb_i+\p_{x_i}\p_tb_j),\p_x^\a\Gamma_{ij}[(\I-\P)g]\r_x}_{B_{41}}+\underbrace{\sum_{i,j=1}^3\l\p_x^\a\up{div}_x b,\p_x^\a\big[(\frac{1}{\v}u_{0}+u)\c b\big]\r_x}_{B_{42}}\\
\underbrace{-\sum_{i,j=1}^3\frac{2}{\v^2}\l\p_x^\a(\p_{x_j}b_i+\p_{x_i}b_j),\p_x^\a\Gamma_{ij}[(\I-\P)g]\r_x}_{B_{43}}+\underbrace{\sum_{i,j=1}^3\l\p_x^\a(\p_{x_j}b_i+\p_{x_i}b_j),\p_x^\a\Gamma_{ij}(l)\r_x}_{B_{44}}
\end{multline}
for $0 \leq |\a| \leq 3$.

Replacing the time derivative of $b$ by $\eqref{Pg-a-b}_2$, $B_{41}$ can be divided into the following parts:
\begin{equation}\label{B51-0}
\begin{aligned}
B_{41}=&\ -2\sum_{i,j=1}^3\l \p_x^\a\p_t b_i, \p_x^\a\p_{x_j}\Gamma_{ij}[(\I-\P)g] \r_x\\
=&\ \underbrace{-\frac{2}{\v}\sum_{i,j=1}^3\l\p_x^\a R_{4i}, \p_x^\a\p_{x_j}\Gamma_{ij}[(\I-\P)g] \r_x}_{B_{411}} \underbrace{-2\sum_{i,j=1}^3\l\p_x^\a R_{5i},\p_x^\a \p_{x_j}\Gamma_{ij}[(\I-\P)g]\r_x}_{B_{412}}\\
&\ +\underbrace{\frac{2}{\v}\sum_{i,j=1}^3 \l\p_x^\a \p_{x_i} a ,\p_x^\a\p_{x_j} \Gamma_{ij}[(\I-\P)g]\r_x}_{B_{413}} +\underbrace{\frac{2}{\v^2}\sum_{i,j=1}^3\l\p_x^\a(b_i-\v u_i),\p_x^\a\p_{x_j}\Gamma_{i,j}[(\I-\P)g]\r_x}_{B_{414}}\\
&\ +\underbrace{\frac{2}{\v}\sum_{i,j,k=1}^3\l\p_x^\a \p_{x_k}(\I-\P)g(v_kv_i-1),\p_x^\a\p_{x_j}\Gamma_{ij}[(\I-\P)g]\M\r_{x,v}}_{B_{415}}.
\end{aligned}
\end{equation}
where, in the last term above, we consider the fact that $\Gamma_{ij}[\p_{x_j}(\I-\P)g]=\l \p_{x_j}(\I-\P)g,\M \r_v$.\\
For $B_{411}$, by recalling the definition of $R_{4i}$ and $\Gamma[(\I-\P)g]$ in \eqref{R4i-R5i-Gamma-l}, for $|\a'|=1$,
\begin{equation*}\label{B511}
\begin{aligned}
|B_{411}| \lesssim \frac{1}{\v}\|u_0 a+u m_0\|_{L^2_x} \|\n_x(\I-\P)g\|_{L^2_{x,v}}
\lesssim&\ \frac{1}{\v}(\|u_0\|_{L^3_x}\|a\|_{L^6_x}
+\|m_0\|_{L^3_x}\| u\|_{L^6_x})\|\n_x(\I-\P)g\|_{L^2_{x,v}}\\
\lesssim &\ \frac{1}{\v}\|(u_0,m_0)\|_{H^1_x} \|(\n_x a,\n_x u)\|_{L^2_x} \|\n_x(\I-\P)g\|_{L^2_{x,v}} \\
\lesssim &\ \E^{\frac{1}{2}}(t)\D(t), \quad \text{for} \ |\a| =0,
\end{aligned}
\end{equation*}
\begin{equation*}\label{B511_2}
\begin{aligned}
|B_{411}| \lesssim& \ \frac{1}{\v}\|\p_x^\a(u_0 a+u m_0)\|_{L^2_x} \|\p_x^{\a+\a'}(\I-\P)g\|_{L^2_{x,v}} \\
\lesssim& \ \frac{1}{\v} \Big(\|u_0\|_{L^\infty_x}\|\p_x^\a a\|_{L^2_x} +\|\p_x^\a u_0\|_{L^3_x}\|a\|_{L^6_x}\\
&\ +\sum_{1\leq|\tilde{\a}|\leq|\a|-1}\|\p_x^{\tilde{\a}}u_0\|_{L^4_x}\|\p_x^{\a-\tilde{\a}} a\|_{L^4_x}+\|m_0\|_{L^\infty_x}\|\p_x^\a u\|_{L^2_x} +\|\p_x^\a m_0\|_{L^3_x}\|u\|_{L^6_x}\\
&\ + \sum_{1\leq|\tilde{\a}|\leq|\a|-1}\|\p_x^{\tilde{\a}}m_0\|_{L^4_x}\|\p_x^{\a-\tilde{\a}} u\|_{L^4_x} \Big) \|\p_x^{\a+\a'}(\I-\P)g\|_{L^2_{x,v}}\\
\lesssim &\ \E^{\frac{1}{2}}(t)\D(t),\quad \text{for} \ 1 \leq |\a| \leq 3,
\end{aligned}
\end{equation*}
where the H$\ddot{\up{o}}$lder inequality in $v$ is applied that
\begin{multline*}
\|\n_x\Gamma_{ij}[(\I-\P)g]\|_{L^2_x}=\|\int\n_x(\I-\P)g(v_i v_j-\delta_{ij})\M dv\|_{L^2_x} \\
\leq \|\n_x(\I-\P)g\|_{L^2_{x,v}}\|(v_iv
_j-\delta_{ij})\M\|_{L^2_v} \lesssim \|\n_x(\I-\P)g\|_{L^2_{x,v}}.
\end{multline*}
Recalling $R_{5i}$ in \eqref{R4i-R5i-Gamma-l}, $B_{412}$ is bounded by, for $|\a'|=1$,
\begin{equation*}\label{B512}
\begin{aligned}
|B_{412}| =&\ \Big| \sum_{i,j=1}^3\l\p_x^\a R_{5i},\p_x^\a \p_{x_j}\Gamma_{ij}[(\I-\P)g]\r_x \Big| \\
\lesssim&\ \v^2\|(\p_t\p_x^{\a}u_0,\p_t\p_x^{\a+\a'}m_0)\|^2_{L^2_x} + \frac{1}{\v^2}\|\p_x^{\a+\a'}(\I-\P)g\|^2_{L^2_{x,v}} + \E^{\frac{1}{2}}(t)\D(t)\\
\lesssim&\ \D_{ma}(t)+\D_{mi,K,1}(t)+\E^{\frac{1}{2}}(t)\D(t).
\end{aligned}
\end{equation*}
By applying the similar argument as $B_{411}$ and $B_{412}$, we have, for $|\a'|=1$,
\begin{equation*}\label{B513-B514-B515}
\begin{aligned}
|B_{413}| + |B_{414}| + |B_{415}| \leq&\ \frac{1}{4}\|\p_x^{\a+\a'}a\|^2_{L^2_x} + \frac{C}{\v^2}\big[\|\p_x^{\a+\a'}(\I-\P)g\|^2_{L^2_x} + \|\p_x^\a(b-\v u)\|^2_{L^2_x} \big]\\
&\ +\frac{2}{\v}\sum_{i,j,k=1}^3\l\p_x^{\a}\p_{x_k}\Gamma_{ik}[(\I-\P)g], \p_x^{\a}\p_{x_i}\Gamma_{ij}[(\I-\P)g]\r_x\\
\leq&\ \frac{1}{4}\|\p_x^{\a+\a'}a\|^2_{L^2_x} + \frac{C}{\v}\|\p_x^{\a+\a'}(\I-\P)g\|^2_{L^2_{x,v}} + C\E^{\frac{1}{2}}(t)\D(t)\\
\leq&\ \frac{1}{4}\|\p_x^{\a+\a'}a\|^2_{L^2_x} + C\D_{mi,K,1}(t)+C\E^{\frac{1}{2}}(t)\D(t).
\end{aligned}
\end{equation*}
Therefore, combing the estimates on $B_{411}$, $B_{412}$, $B_{413}$, $B_{414}$, $B_{415}$, and substitute them into \eqref{B51-0}, we have, for $|\a'|=1$,
\begin{equation}\label{B51}
\begin{aligned}
|B_{41}| \leq \ \frac{1}{4}\|\p_x^{\a+\a'}a\|^2_{L^2_x} + C \big(\D_{ma}(t)+\D_{mi,K,1}(t)+\E^{\frac{1}{2}}(t)\D(t)\big).
\end{aligned}
\end{equation}


For $B_{42}$ and $B_{43}$, we also have, for $|\a'|=1$,
\begin{equation}\label{B52-B53}
\begin{aligned}
|B_{42}| + |B_{43}|
\lesssim&\ \frac{1}{\v}\|\p_x^{\a+\a'}b\|_{L^2_x} \|u_0\|_{H^{|\a|}_x}\|b-\v u\|_{H^{|\a|}_x} + \|\p_x^{\a+\a'}b\|_{L^2_x} \|u\|_{H^{|\a|}_x} \Big(\|\n_x u_0\|_{H^{|\a|-1}_x}+\|\n_x b\|_{H^{|\a|-1}_x}\Big)\\
\lesssim&\ \frac{1}{\v^2}\big[\|\p_x^{\a+\a'}(b-\v u)\|^2_{L^2_x} + \|\p_x^\a(\I-\P)g\|^2_{L^2_{x,v}} \big] + \|\p_x^{\a+\a'}u\|^2_{L^2_x}\\
\lesssim&\ \E^{\frac{1}{2}}(t)\D(t)+ \D_{mi,K,1}(t).
\end{aligned}
\end{equation}


Recalling the definition of $l$ in \eqref{R4i-R5i-Gamma-l}, $B_{44}$ is divided into four parts as follows:
\begin{equation*}\label{B54-0}
\begin{aligned}
B_{44}=&\underbrace{-\frac{1}{\v}\sum_{i,j=1}^3 \l\p_x^\a(\p_{x_i} b_j+\p_{x_j} b_i), \p_x^\a[(\I-\P_0)v\c\n_x g_1](v_iv_j-1)\M\r_{x,v}}_{B_{441}}\\
&\underbrace{-\frac{1}{\v}\sum_{i,j=1}^3 \l\p_x^\a(\p_{x_i} b_j+\p_{x_j} b_i),\p_x^\a [v\c\n_x(\I-\P)g](v_iv_j-1)\M\r_{x,v}}_{B_{442}}\\
&\underbrace{-\frac{1}{\v}\sum_{i,j=1}^3 \l\p_x^\a(\p_{x_i} b_j+\p_{x_j} b_i),\p_x^\a \up{div}_v(u_0 g_1\M)(v_iv_j-1)\r_{x,v}}_{B_{443}}\\
&\underbrace{-\sum_{i,j=1}^3\l\p_x^\a(\p_{x_i} b_j+\p_{x_j} b_i),\p_x^\a \up{div}_v(u g_1 \M)(v_iv_j-1)\r_{x,v}}_{B_{444}}.
\end{aligned}
\end{equation*}
For $B_{441}$, by recalling the definition of $g_1$, it can be bounded by
\begin{equation*}\label{B541}
|B_{441}| \lesssim  \D_{ma}(t)+\D_{mi,K,1}(t)+\E^{\frac{1}{2}}(t)\D(t).
\end{equation*}
For $B_{442}$ and $B_{443}$, we can apply the similar argument as $B_{43}$ and obtain
\begin{equation*}\label{B542-B543}
\begin{aligned}
|B_{442}| + |B_{443}| \lesssim \ \D_{mi,K,1}(t)+\E^{\frac{1}{2}}(t)\D(t).
\end{aligned}
\end{equation*}
For $B_{444}$, we find, for $|\a'|=1$,
\begin{equation*}\label{B544}
\begin{aligned}
|B_{444}| = & \sum_{i,j=1}^3\l\p_x^\a(\p_{x_i} b_j+\p_{x_j} b_i),\p_x^\a(u_j g_1 v_i+u_i g_1 v_j)\r_{x,v}\\
\lesssim &\ \|\p_x^{\a+\a'}b\|_{L^2_x} \|\p_x^\a(u_j g_1 v_i+u_i g_1 v_j)\|_{L^2_{x,v}}\\
\lesssim &\ \E^{\frac{1}{2}}(t)\D(t).
\end{aligned}
\end{equation*}
Hence, collecting the estimates on $B_{441}$, $B_{442}$, $B_{443}$ and $B_{444}$, we have
\begin{equation}\label{B54}
\begin{aligned}
|B_{44}| \lesssim \D_{ma}(t)+\D_{mi,K,1}(t)+\E^{\frac{1}{2}}(t)\D(t).
\end{aligned}
\end{equation}


Noticing that, for any $0\leq|\a|\leq 3$ and $|\a'|=1$,
\begin{equation}\label{partial-b-0}
\frac{1}{\v}\sum_{i,j=1}^3\|\p_x^\a(\p_{x_i} b_j+\p_{x_j} b_i)\|^2_{L^2_x} = \frac{2}{\v}(\|\p_x^{\a+\a'}b\|^2_{L^2_x} + \|\p_x^\a\up{div}_x b\|^2_{L^2_x} ),
\end{equation}
and substituting the estimates on $B_{41}$, $B_{42}$, $B_{43}$, $B_{44}$ into\eqref{partial-b-1}, it follows that, for $|\a'|=1$,
\begin{multline}\label{B5}
\frac{d}{dt}\sum_{i,j=1}^3\int\p_x^\a(\p_{x_i} b_j + \p_{x_j} b_i)\c\p_x^\a\Gamma_{ij}[(\I-\P)g] \,dx + \frac{2}{\v}(\|\p_x^{\a+\a'}b\|^2_{L^2_x} + \|\p_x^\a\up{div}_x b\|^2_{L^2_x})\\
\leq \frac{1}{4}\|\p_x^{\a+\a'}a\|^2_{L^2_x} + C\big(\D_{ma}(t)+\D_{mi,K,1}(t)+\E^{\frac{1}{2}}(t)\D(t)\big).
\end{multline}


On the other hand, considering the equation $\eqref{Pg-a-b}_2$, we have,
\begin{equation}\label{B6-0}
\begin{aligned}
\|\p_x^{\a+\a'}a\|^2_{L^2_x}= &\sum_{i=1}^3\l\p_x^\a\p_{x_i} a,\p_x^\a\p_{x_i} a \r_x=-\v\frac{d}{dt} \int \partial_x^{\a+\a'} a\c\p_x^\a b dx + \underbrace{\v\sum_{i=1}^3\l\p_x^\a\p_{x_i} \p_t a,\p_x^\a b_i\r_x}_{B_{51}}\\
&\underbrace{-\frac{1}{\v}\sum_{i=1}^3\l\p_x^\a\p_{x_i} a,\p_x^\a(b_i-\v u_i)\r_x}_{B_{52}} \underbrace{-\sum_{i,j=1}^3\l\p_x^\a\p_{x_i} a,\p_x^\a\p_{x_i}(\I-\P)g(v_iv_j-1)\M\r_{x,v}}_{B_{53}}\\
& +\underbrace{\sum_{i=1}^3 \l\p_x^\a\p_{x_i} a,\p_x^\a R_{4i}\r_x}_{B_{54}} + \underbrace{\sum_{i=1}^3\l\p_x^\a\p_{x_i} a,\p_x^\a R_{5i}\r_x}_{B_{55}}
\end{aligned}
\end{equation}
for $0 \leq |\a| \leq 3$ and $|\a'|=1$.\\
For $B_{51}$, by noting $\v\p_t a=-\up{div}_x b$ in $\eqref{Pg-a-b}_1$, we have
\begin{equation*}\label{B61}
B_{51} = -\sum_{i=1}^3\l\p_x^\a \p_{x_i} \up{div}_x b,\p_x^\a b_i\r= \|\p_x^\a\up{div}_x b\|^2_{L^2_x}.
\end{equation*}
For $B_{52}$ and $B_{53}$, we find, for $|\a'|=1$,
\begin{equation*}\label{B62}
\begin{aligned}
|B_{52}|=& \ \Big| -\frac{1}{\v}\sum_{i=1}^3\l\p_x^\a\p_{x_i} a, \p_x^\a(b_i-\v u_i)\r_x \Big| \\
\leq&\ \frac{1}{16}\|\p_x^{\a+\a'} a\|^2_{L^2_x} + \frac{C}{\v^2}\|\p_x^\a(b-\v u)\|^2_{L^2_x} \\
\leq&\ \frac{1}{16}\|\p_x^{\a+\a'} a\|^2_{L^2_x} + C\D_{mi,K,1}(t),
\end{aligned}
\end{equation*}
\begin{equation*}\label{B63}
\begin{aligned}
|B_{53}| \leq&\ \frac{1}{16}\|\p_x^{\a+\a'} a\|^2_{L^2_x} + \frac{C}{\v^2}\|\p_x^{\a+\a'}(\I-\P)g\|^2_{L^2_x} \\
\leq&\ \frac{1}{16}\|\p_x^{\a+\a'} a\|^2_{L^2_x} + C\D_{mi,K,1}(t).
\end{aligned}
\end{equation*}
Recalling $R_{4i}$ and $R_{5i}$ in \eqref{R4i-R5i-Gamma-l}, we have, for $|\a'|=1$,
\begin{equation*}\label{B64}
\begin{aligned}
|B_{54}|=&\ \Big|\sum_{i=1}^3\l\p_x^\a\p_{x_i} a,\p_x^\a(u_{0i}a+u_i m_0)\r_x \Big| \\
\lesssim &\ \|\p_x^{\a+\a'}a\|_{L^2_x} \|u_0\|_{H^{|\a|}_x}\|\n_x a\|_{H^{|\a|-1}_x} + \|\p_x^{\a+\a'}a\|_{L^2_x} \|m_0\|_{H^{|\a|}_x}\|\n_x u\|_{H^{|\a|-1}_x}\\
\lesssim &\ \E^{\frac{1}{2}}(t)\D(t),
\end{aligned}
\end{equation*}
\begin{equation*}\label{B65}
\begin{aligned}
|B_{55}| =&\ \Big|-\sum_{i=1}^3\l\p_x^\a\p_{x_i} a,\p_x^\a\big[\p_t(u_{0i} m_0)+u_i a+\p_t u_{0i}+\p_t\p_{x_i} m_0\big] \r_x \Big| \\
\lesssim&\ \|\p_x^{\a+\a'}a\|_{L^2_x} \big(\|\p_t u_0\|_{H^{|\a|}_x}\|\n_x m_0\|_{H^{|\a|-1}_x} + \|\p_t m_0 \|_{H^{|\a|}_x}\|\n_x u_0\|_{H^{|\a|-1}_x}\big)\\[2pt]
&+ \|\p_x^{\a+\a'}a\|_{L^2_x} \|\p_t\p_x^\a u_0\|_{L^2_x} \|\p_t\p_x^{\a+\a'} m_0\|_{L^2_x}\\
\leq&\ \frac{1}{16}\|\p_x^{\a+\a'} a\|^2_{L^2_x} + C(\|\p_t\p_x^\a u_0\|^2_{L^2_x} + \|\p_t\p_x^{\a+\a'} m_0\|^2_{L^2_x} ) + C\E^{\frac{1}{2}}(t)\D(t)\\
\leq&\ \frac{1}{16}\|\p_x^{\a+\a'} a\|^2_{L^2_x} + C\big(\D_{ma}(t)+\E^{\frac{1}{2}}(t)\D(t)\big).
\end{aligned}
\end{equation*}
Then, by inserting the estimates of $B_{51}$, $B_{52}$, $B_{53}$, $B_{54}$ and $B_{55}$ into \eqref{B6-0}, we obtain, for $|\a'|=1$,
\begin{equation}\label{B6}
\begin{aligned}
\frac{7}{8}\|\p_x^{\a+\a'}a\|^2_{L^2_x} + \v\frac{d}{dt}\int\p_x^{\a+\a'} a\c\n_x^\a b \, dx\leq C\big(\D_{ma}(t)+\D_{mi,K,1}(t)+\E^{\frac{1}{2}}(t)\D(t)\big).
\end{aligned}
\end{equation}

Thus, combining \eqref{Energy of a and b}, \eqref{B5} and \eqref{B6}, we have,
\begin{multline}\label{B5-B6}
\frac{d}{dt} \int\p_x^\a(\p_{x_i} b_j+\p_{x_j} b_i)\c\p_x^\a\Gamma_{ij}[(\I-\P)g] \,dx + \v\frac{d}{dt}\int\p_x^{\a+\a'} a\c\p_x^\a b \,dx \\
+\frac{1}{2\v} \big(\|\p_x^{\a+\a'}b\|^2_{L^2_x} + \|\p_x^\a\up{div}_x b\|^2_{L^2_x} \big) + \frac{1}{2}\|\p_x^{\a+\a'}a\|^2_{L^2_x}
\leq C\big(\D_{ma}(t)+\D_{mi,K,1}(t)+\E^{\frac{1}{2}}(t)\D(t)\big),
\end{multline}
for any $0\leq |\a|\leq 3$ and $|\a'|=1$.

The proof of Lemma \ref{Estiamte-Mi-fluid-part} is finally completed by summing up $0 \leq |\a| \leq 3$ in \eqref{B5-B6}.

\end{proof}

\subsection{Proof of Proposition \ref{total_energy_dissipation}}
\label{subsec:proof_total_energy}

In this subsection, we present the specific proof of the total energy estimate Proposition \ref{total_energy_dissipation} by combing the Lemmas \ref{Estimate-kinetic-Mi-1}, \ref{Estiamte-Mi-kinetic-part-mixed-partial} and \ref{Estiamte-Mi-fluid-part} together.

\begin{proof}
We prove the total energy estimate \eqref{total_energy_estimate} in the following three steps:

\textbf{Step 1:} Choose a constant $\tilde{\lambda}_1>0$ small enough such that
\begin{equation}\label{tilde-lambda-1}
C_2 \tilde{\lambda}_1 \leq \frac{C_{1,3}}{2},
\end{equation}
where $C_2$,\,$C_{1,3}$ are constants in inequality \eqref{Step two} of Lemma \ref{Estiamte-Mi-kinetic-part-mixed-partial} and in inequality \eqref{Last step} of Lemma \ref{Estiamte-Mi-fluid-part} respectively.
Then, by applying \eqref{Last step} $ + \tilde{\lambda}_1\times$ \eqref{Step two}, we can find that there exists constants $C_{1,4}\,C_4>0$ such that
\begin{multline}\label{First time}
    \frac{1}{2}\frac{d}{dt}\big(\E_{mi,F}(t)+\tilde{\lambda}_1\E_{mi,K,2}(t)\big)+C_{1,4}\big(\D_{mi,K,2}(t)+\D_{mi,F}(t)\big) \\
    \leq C_4 \big(\D_{mi,K,1}(t)+\D_{ma}(t)+\E^{\frac{1}{2}}(t)\D(t)\big).
\end{multline}

\textbf{Step 2:} Choose a constant $\tilde{\lambda}_2>0$ small enough such that
\begin{equation}\label{tilde-lambda-2}
C_4 \tilde{\lambda}_2 \leq \frac{C_{1,1}}{2},
\end{equation}
where $C_{1,1}$ is the constant in \eqref{Step one} and $C_4$ is the constant in \eqref{First time}.
Then, by applying \eqref{Step one}$+\tilde{\lambda}_2\times$\eqref{First time}, there exists constants $C_{1,5}, C_5 >0$ such that
\begin{multline}\label{Third time}
\frac{1}{2}\frac{d}{dt}\big(\E_{mi,K,1}(t)+\tilde{\lambda}_1\tilde{\lambda}_2\E_{mi,K,2}(t)+\tilde{\lambda}_2\E_{mi,F}(t)\big)+C_{1,5}\big(\sum_{i=1}^2\D_{mi,K,i}(t)+\D_{mi,F}(t)\big)\\
\leq C_5 \big(\D_{ma}(t)+\E^{\frac{1}{2}}(t)\D(t)\big).
\end{multline}

\textbf{Step 3:} Choose a constant $\tilde{\lambda}_3>0$ small enough such that

\begin{equation}\label{tilde-lambda-4}
C_5 \tilde{\lambda}_3 \leq \frac{1}{4},
\end{equation}
where $C_5$ is a constant in \eqref{Third time}.
Therefore, by using \eqref{Third time}$+\tilde{\lambda}_3\times$\eqref{Third time}, there exists a constant $C_{1,6}\,C_6>0$ such that
\begin{multline}\label{Last time}
    \frac{1}{2}\frac{d}{dt}\big(\E_{ma}(t) + \tilde{\lambda}_3\E_{mi,K,1}(t) + \tilde{\lambda}_1\tilde{\lambda}_2\tilde{\lambda}_3\E_{mi,K,2}(t) + \tilde{\lambda}_2\tilde{\lambda}_3\E_{mi,F}(t)\big)\\
 + C_{1,6}\big(\sum_{i=1}^2\D_{mi,K,i}+\D_{mi,F}(t)+\D_{ma}(t)\big) \leq C_6\E^{\frac{1}{2}}(t)\D(t).
\end{multline}

Denoting
\begin{equation}\label{Energy functional}
\begin{aligned}
\E(t):=&\E_{ma}(t)+\tilde{\lambda}_3\E_{mi,K,1}(t)+\tilde{\lambda}_1\tilde{\lambda}_2\tilde{\lambda}_3\E_{mi,K,2}(t)+\tilde{\lambda}_2\tilde{\lambda}_3\E_{mi,F}(t),\\
\D(t):=&\sum_{i=1}^2 \D_{mi,K,i}+\D_{mi,F}(t)+\D_{ma}(t).
\end{aligned}
\end{equation}
and re-naming $\tilde{C}=C_{1,6}$, $C=C_{6}$, we finally obtain the total energy estimate \eqref{total_energy_estimate}:
\begin{equation}
\begin{aligned}
\frac{1}{2}\frac{d}{dt}\E(t)+\tilde{C}\D(t)\leq C\E^{\frac{1}{2}}(t)\D(t),
\end{aligned}
\end{equation}
where $\lambda_i, i =1,2,3,4$, defined as follows:
\begin{equation}\label{lambda constants}
\quad \lambda_1=\tilde{\lambda}_2\tilde{\lambda}_3,
\quad
\lambda_2=\tilde{\lambda}_3,
\quad
\lambda_3=\tilde{\lambda}_1\tilde{\lambda}_2\tilde{\lambda}_3,
\quad
\lambda_4=1.
\end{equation}
\end{proof}






\section{Proof of main result}
\label{sec:proof_main}

\subsection{Proof of Theorem \ref{Main-Limits}}
\label{subsec:proof_main_theorem}

In this subsection, we first give a proof of Theorem \ref{Main-Limits}.
Define
\begin{equation}\label{T-*}
T^* = \sup\{\tau\geq 0\,|\,C\sup_{0\leq t\leq \tau}\mathbb{E}^{\frac{1}{2}}(t)\leq \frac{1}{2}\},
\end{equation}
where $C$ is given in \eqref{instant energy functional-energy dissipative rate functional}.

Considering $\mathbb{E}(0) \sim \E(0)$ and the initial condition in Theorem \ref{Main-Limits}, we can find a constant $C^* > 0$ such that
\begin{equation}\label{initialdata-instant energy}
2C\mathbb{E}(0)\leq C^* \E(0)\leq C^* \delta^*,
\end{equation}
where $C$ is given in \eqref{instant energy functional-energy dissipative rate functional} and $ \delta^*$ is chosen to be $0 < \delta^* \leq \min\{1,\frac{1}{16 C^* C^2}\}$.
Then, for any $0<\v\leq 1$,
\begin{equation}\label{initial-delta-1}
\begin{aligned}
C\mathbb{E}^{\frac{1}{2}}(0)\leq C\sqrt{C^* \delta^*}\leq \frac{1}{4}<\frac{1}{2},
\end{aligned}
\end{equation}
where $C,\,C^*$ are given in \eqref{initialdata-instant energy}

Then the continuity of $\mathbb{E}(t)$ implies that $T^*>0$. Consequently, we derive from the definition of $T^*$, \eqref{instant energy functional-energy dissipative rate functional} and \eqref{initialdata-instant energy} that, for any $t\in[0, T^*]$, we find
\begin{equation}
\mathbb{E}(t)+\int_0^t\mathbb{D}(s) \,ds \leq 2C\mathbb{E}(0) \leq C^* \delta^*,
\end{equation}
uniformly for all $0<\v\leq 1$, which is actually the energy estimate \eqref{The global estimate} and immediately implies by the initial bound \eqref{initial-delta-1} that,
\begin{equation}
C\sup_{0\leq \tau\leq t}\mathbb{E}^{\frac{1}{2}}(\tau) \leq \frac{1}{4} < \frac{1}{2},
\end{equation}
for any $0<\v\leq 1$. Thus, it follows the continuity of $\mathbb{E}(t)$ and the definition of $T^*$ that $T^*=+\infty$.


\subsection{Proof of Corollary \ref{Conergence}}
\label{subsec:proof_main_corollary}

Recalling $g=(\I-\P)g+\P g=(\I-\P)g+(a+v\c b)\M$ and $f^\v=(1+m_0)M + \v[(v\c u_0)(1+m_0)-v\c\n_x m_0\big]M + \v g\M,\, u^\v=u_0+\v u$, and applying the Sobolev embedding $H^2\hookrightarrow L^\infty$, we have
\begin{equation}
\begin{aligned}
|f^\v-(1+m_0)M|\leq&\sup_{t\geq 0}\|f^\v-(1+m_0)M\|_{L^\infty_xL^\infty_v}\\
\lesssim &\ \sup_{t\geq 0}\|f^\v-(1+m_0)M\|_{H^2_x H^2_v} \\
\lesssim &\ \v \sup_{t\geq 0}\|[(v\c u_0)(1+m_0)-v\c\n_x m_0\big]M+g\M\|_{H^2_x H^2_v}\\
\lesssim &\ \v\sup_{t\geq 0}\|(v\c u_0)(1+m_0)-v\c\n_x m_0\big]M\|_{H^2_xH^2_v} +  \v \sup_{t\geq 0} \|g \sqrt{M} \|_{H^2_x L^2_v}\\
&\ + \v\sup_{t\geq 0}\|\n_v\big[(\I-\P)g\M\big]\|_{H^2_x H^1_v} + \v\sup_{t\geq 0} \|\n_v\big[\P g \sqrt{M}\big]\|_{H^2_x H^1_v}\\
\lesssim &\  \v\sup_{t\geq 0}\big[\|u_0\|_{H^2_x}(1+\|m_0\|_{H^2_x})+\|\n_xm_0\|_{H^2_x} + \|g \|_{H^2_x L^2_v}\\
&\ +\|\n_v(\I-\P)g\|_{H^2_xH^1_v}+\|(a,b)\|_{H^2_x}\big] \lesssim \v, \\[4pt]
|u^\v-u_0| \leq& \ \v\sup_{t\geq 0}\|u\|_{L^\infty_x}\lesssim \v\sup_{t\geq 0}\|u\|_{H^2_x}\lesssim \v,
\end{aligned}
\end{equation}
where the uniform boundedness of all quantities is guaranteed by the global-in-time energy estimate \eqref{The global estimate}
\cite{CYP17}.

\section*{Acknowledgements}
We would like to thank the anonymous reviewers for their valuable comments and suggestions, which helped us to improve the quality of the manuscript. 
Most parts of the work were done when K. Qi visited the School of Mathematics at the South China University of Technology during the summer of 2023. He would like to thank the department for its hospitality. This work was partially supported by the National Natural Science Foundation of China (12071152), by the Guangdong Basic and Applied Basic Research Foundation (2022A1515012112, 2020B1515310015), and by the Guangdong Provincial Key Laboratory of Human Digital Twin (2022B1212010004).

\bibliographystyle{siam}
\bibliography{Kinetic-Fluid_Project}

\end{document}